\documentclass[10pt,a4paper,openany,oneside]{article}
\usepackage{times}
\usepackage{authblk}
\usepackage[latin1]{inputenc}
\usepackage[english]{babel}
\usepackage{mathrsfs}
\usepackage{stmaryrd}
\usepackage{amssymb,amsmath,amsthm}
\usepackage{hyperref}
\usepackage[dvipsnames]{xcolor}
\usepackage{graphicx}
\usepackage{subcaption}
\usepackage{eurosym}
\usepackage{color}
\usepackage{enumerate}
\usepackage{multirow}
\usepackage{stmaryrd}
\usepackage{booktabs}
\usepackage{float}
\theoremstyle{plain}
\newtheorem{thm}{Theorem}[section]

\newtheorem{lemma}[thm]{Lemma}
\newtheorem{prop}[thm]{Proposition}

\newtheorem{example}{Example}
\theoremstyle{definition}

\theoremstyle{remark}
\newtheorem{rem}{Remark}
\newcommand{\C}{\mathbb C}
\newcommand{\R}{\mathbb R}
\newcommand{\N}{\mathbb N}

\newcommand{\PT}{\mathcal P\mathcal T}

\def\eps{\varepsilon}

\def\ri{{\rm i}}
\def\bi{\begin{itemize}}
	\def\ei{\end{itemize}}

\newcommand{\B}{\mathcal{B}}

\newcommand{\Z}{\mathbb{Z}}

\newcommand{\cH}{{\mathcal H}}

\newcommand{\cL}{{\mathcal L}}
\newcommand{\cM}{{\mathcal M}}

\renewcommand{\dim}{{\rm dim}\,}

\newcommand{\pa}{\partial}

 \def\dd{\, {\rm d}}

\DeclareMathOperator{\spann}{span}

\DeclareMathOperator{\Real}{Re}
\DeclareMathOperator{\Imag}{Im}
\DeclareOldFontCommand{\it}{\normalfont\itshape}{\mathit}

\newcommand{\bspm}{\left(\begin{smallmatrix}}\newcommand{\espm}{\end{smallmatrix}\right)}
\newcommand{\bpm}{\begin{pmatrix}}\newcommand{\epm}{\end{pmatrix}}
\newcommand{\trans}{{\mathsf T}}
\def\blem{\begin{lemma}}\def\elem{\end{lemma}}
\def\bthm{\begin{theorem}}\def\ethm{\end{theorem}}
\def\bcor{\begin{corollary}}\def\ecor{\end{corollary}}
\def\beq{\begin{equation}}\def\eeq{\end{equation}}
\def\beqq{\begin{equation*}}\def\eeqq{\end{equation*}}
\def\bal{\begin{align}}\def\eal{\end{align}}
\def\bpf{\begin{proof}}\def\epf{\end{proof}}
\def\bex{\begin{example}}\def\eex{\end{example}}
\def\brem{\begin{remark}}\def\erem{\end{remark}}
\def\bass{\begin{assumption}}\def\eass{\end{assumption}}
\def\bprop{\begin{proposition}}\def\eprop{\end{proposition}}
\def\bdefi{\begin{definition}}\def\edefi{\end{definition}}

\begin{document}
	\title{Bifurcation and Asymptotics of Cubically Nonlinear Transverse Magnetic Surface Plasmon Polaritons}
	\date{\today}
	\author{ Tom\'a\v{s} Dohnal$^*$ and Runan He   }
	\affil{Institut f\"{u}r Mathematik,  Martin-Luther-Universit\"{a}t Halle-Wittenberg,\\ 06120 Halle (Saale), Germany\\
		\small{tomas.dohnal@mathematik.uni-halle.de$^*$, runan.he@mathematik.uni-halle.de}}

	\maketitle
	\abstract{Linear Maxwell equations for transverse magnetic (TM)  polarized fields support single frequency surface plasmon polaritons (SPPs) localized at the interface of a metal and a dielectric. Metals are typically dispersive, i.e. the dielectric function depends on the frequency. We prove the bifurcation of localized SPPs in dispersive media in the presence of a cubic nonlinearity and provide an asymptotic expansion of the solution and the frequency. The problem  is reduced to a system of nonlinear differential equations in one spatial dimension by assuming a plane wave dependence in the direction tangential to the (flat) interfaces. The number of interfaces is arbitrary and the nonlinear  system is solved in a subspace of functions with the $H^1$-Sobolev regularity in each material layer. The corresponding linear system is an operator pencil in the frequency parameter due to the material dispersion. 
	The studied bifurcation occurs at a simple isolated eigenvalue of the pencil. For geometries consisting of two or three homogeneous layers we provide explicit conditions on the existence of eigenvalues and on their simpleness and isolatedness. Real frequencies are shown to exist in the nonlinear setting in the case of $\PT$-symmetric materials. We also apply a finite difference numerical method  to the nonlinear system and compute bifurcating curves.
	}

	\vskip0.2truecm

	\section{Introduction}\label{S:intro}
	In this article we study time harmonic electromagnetic waves at one or more interfaces between layers of nonlinear and dispersive media. In applications these are typically layers of dielectric and metallic materials and the waves are referred to as surface plasmon polaritons (SPPs). The underlying model is given by Maxwell's equations with the absence of free charges, i.e.
	\begin{equation}\label{MaxwellEq}
		\partial_t\mathcal{D}=\nabla\times\mathcal{H},\quad	\mu_0\partial_t\mathcal{H}=-\nabla\times\mathcal{E},\quad \nabla\cdot\mathcal{D}=\nabla\cdot\mathcal{H}=0,
	\end{equation}
	where $\mathcal{E}$ and $\mathcal{H}$ are the electric and magnetic field respectively, $\mathcal{D}$ is the electric displacement field depending on $\mathcal{E}$ in a Kerr nonlinear and nonlocal relation
	\begin{equation}\label{Dispalacement}
		\begin{split}
			&\mathcal{D}(x,y,z,t)=\epsilon_0\mathcal{E}(x,y,z,t)+\epsilon_0\int_{\mathbb{R}}\chi^{(1)}\left(x, y, z, t-s\right)\mathcal{E}(x,y,z,s)\mathrm{d}s\\
			&\quad +\epsilon_0\int_{\mathbb{R}^3}\chi^{(3)}\left(x, y, z, t-s_1, t-s_2, t-s_3\right)\left(\left(\mathcal{E}(x,y,z,s_2)\cdot\mathcal{E}(x,y,z,s_3)\right)\mathcal{E}(x,y,z,s_1)\right)\mathrm{d}(s_1,s_2,s_3),
		\end{split}
	\end{equation}
	$\chi^{(1)}:\mathbb{R}^4\to\mathbb{R}$, $\chi^{(3)}:\mathbb{R}^6\to\mathbb{R}$, $\chi^{(1)}(\cdot,\tau)=0$ for $\tau<0$, and $\chi^{(3)}(\cdot,\tau_1,\tau_2,\tau_3)=0$ if $\tau_1<0$ or $\tau_2<0$ or $\tau_3<0$.
	The constants $\epsilon_0$ and $\mu_0$ are the permittivity and the permeability of the free space respectively.

	For a monochromatic field
	$$
	\left(\mathcal{E},\ \mathcal{H},\ \mathcal{D}\right)(x,y,z,t)=\left(E,\ H,\ D\right)(x,y,z)e^{-\mathrm{i}\omega t}+\left(\overline{E},\ \overline{H},\ \overline{D}\right)(x,y,z)e^{\mathrm{i}\omega t},
	$$
	with a real frequency $\omega$, one can obtain a nonlinear eigenvalue problem in $\left(\omega,\ \left(E,\ H\right)\right)$ by neglecting higher harmonics (terms proportional to $e^{3\mathrm{i}\omega t}$ and $e^{-3\mathrm{i}\omega t}$), see \cite{shen1984},
	\begin{subequations}\label{TimeHarmonicME}
		\begin{equation}\label{TimeHarmonicME-1}
			\nabla\times H=-\mathrm{i}\omega D,\quad \nabla\times E=\mathrm{i}\omega\mu_0 H,\quad \nabla\cdot D=\nabla\cdot H=0
		\end{equation}
		\begin{equation}\label{TimeHarmonicME-2}
			\begin{split}
			D=~&\epsilon_0\left(1+\hat{\chi}^{(1)}(x,y,z,\omega)\right)E \\
			+~&\epsilon_0\left((\hat{\chi}^{(3)}(x,y,z,-\omega,\omega,\omega)+\hat{\chi}^{(3)}(x,y,z,\omega,-\omega,\omega))|E|^2E+ \hat{\chi}^{(3)}(x,y,z,\omega,\omega,-\omega)(E\cdot E)\overline{E}\right).
		\end{split}
		\end{equation}
	\end{subequations}
	Here $|E|^2=E\cdot\overline{E}$ and $\hat{f}$ is the Fourier transform of $f$ in time, $\hat{f}(\omega):=\int_\R f(t)e^{\ri \omega t}\dd t$. Clearly, if $\omega \neq 0$, then all solutions of the first two equations in \eqref{TimeHarmonicME-1} satisfy the last two equations, i.e. the divergence conditions.

	For notational simplicity we assume
	$$\hat{\chi}^{(3)}(x,y,z,-\omega,\omega,\omega)=\hat{\chi}^{(3)}(x,y,z,\omega,-\omega,\omega)=\hat{\chi}^{(3)}(x,y,z,\omega,\omega,-\omega)=:\hat{\chi}^{(3)}(x,y,z,\omega).$$
	Then \eqref{TimeHarmonicME-2} becomes
	$$	D=\epsilon_0\left(1+\hat{\chi}^{(1)}(x,y,z,\omega)\right)E+\epsilon_0\hat{\chi}^{(3)}(x,y,z,\omega)\left(2|E|^2E+(E\cdot E)\overline{E}\right).$$
Note that the analysis can be carried out in the same way without this simplification.

\begin{rem}
	The removal of higher harmonics occurs automatically if one uses a time averaged model for the nonlinear part of the displacement field tailored for the frequency $\omega \in \R$. This model has been used, for instance, in \cite{MR2022, Stuart_93, Sutherland03}.  In detail, one sets
		\begin{equation}\label{D-t-avg}
			\begin{split}
				\mathcal{D}(x,y,z,t)&=\epsilon_0\mathcal{E}(x,y,z,t)+\epsilon_0\int_{\mathbb{R}}\chi^{(1)}\left(x, y, z, t-s\right)\mathcal{E}(x,y,z,s)\mathrm{d}s\\
				&+\epsilon_0\chi^{(3)}\left(x,y,z\right)\langle |\mathcal{E}|^2\rangle_\omega(x,y,z)\mathcal{E}(x,y,z,t),
			\end{split}
		\end{equation}
		where $\langle f \rangle_\omega :=\frac{\omega}{2\pi}\int_0^\frac{2\pi}{\omega}f(t)\dd t$. With this model and the monochromatic ansatz one obtains $D=\epsilon_0\left(1+\hat{\chi}^{(1)}(x,y,z,\omega)\right)E+\epsilon_0\chi^{(3)}(x,y,z)|E|^2E$.
\end{rem}

\bigskip
	We consider  one-dimensional structures, where $\hat{\chi}^{(1)}$ and $\hat{\chi}^{(3)}$ are independent of $y$ and $z$, i.e.
	$$\hat{\chi}^{(j)}(x,y,z,\omega)= \hat{\chi}^{(j)}(x,\omega), \ j =1,3,$$
	and study the transverse magnetic fields
	\begin{equation}\label{TM-setting}
		\begin{split}
			E(x,y,z)&=\left(\mu_0u_1(x),  \mu_0u_2(x),  0\right)^Te^{\mathrm{i}ky},\\
			H(x,y,z)&=\left(0,  0,  u_3(x)\right)^Te^{\mathrm{i}ky}
		\end{split}
	\end{equation}
	with $k\in\R$ fixed. The factor $\mu_0$ has been added for convenience. This leads to the vectorial nonlinear system
	\begin{equation}\label{system}
		L_k(x,\omega)u:=Au-B(x,\omega)u=h(x,\omega, u), \quad u:=\left(u_1, u_2, u_3\right)^T, \quad x\in \R,
	\end{equation}
	where \begin{equation}\label{operator}
		\begin{split}
			A:=\begin{pmatrix}
				0\ & 0 \ &\mathrm{i}k \\
				0\ & 0 \ &-\partial_{x} \\
				-\mathrm{i}k\ &\partial_x \ &0
			\end{pmatrix}=\nabla_k\times,\quad \nabla_k=\left(\partial_x,\ \mathrm{i}k,\ 0\right)^T, \qquad \qquad \\
			B(x,\omega):=\begin{pmatrix}
				\ri V(x,\omega)\ &0 \ &0\\
				0\ & \ri V(x,\omega) &0\\
				0\ &0\ & \ri\omega
			\end{pmatrix},\quad
			V(x,\omega):=-\omega \epsilon_0\mu_0\left(1+\hat{\chi}^{(1)}(x,\omega)\right),\\
			h(x, \omega, u):=-\ri \omega \epsilon_0\mu_0^3\hat{\chi}^{(3)}(x,\omega)\left(2|u_E|^2u_E+(u_E\cdot u_E)\overline{u_E}\right),\qquad\qquad\qquad\quad  \end{split}\end{equation}
	and with $u_E := (u_1,u_2,0)^T$.
	In this reduction, solutions $u$ of system \eqref{system} satisfy $\nabla_k \cdot (B(\cdot,\omega)u+h(\cdot,\omega,u))=0$, i.e. $\partial_x D_1 + \partial_y D_2 =0$. Due to the $z$-independence of $E$ and $H$ and because $H_1=H_2=0$, we obtain the divergence conditions $\nabla\cdot D=\nabla \cdot H=0$ as expected. The operator $L_k(x,\omega)$ and the matrix $B(x,\omega)$ will often be abbreviated via $L_k(\omega)$ and $B(\omega)$.

	\begin{rem}\label{rem:cplx-om}
		Note that problem \eqref{system} can be considered also for $\omega=\omega_R +\ri \omega_I \in \C$. However, the resulting field
	$$
	\left(\mathcal{E},\ \mathcal{H},\ \mathcal{D}\right)(x,y,z,t)=\left(E,\ H,\ D\right)(x,y,z)e^{-\mathrm{i}\omega t}+\left(\overline{E},\ \overline{H},\ \overline{D}\right)(x,y,z)e^{\mathrm{i}\overline{\omega} t},
	$$
	solves the nonlinear Maxwell equations \eqref{MaxwellEq}, \eqref{Dispalacement} (with the harmonics $e^{-3 \ri \omega_R t}$ and $e^{3 \ri \omega_R t}$ neglected) or \eqref{MaxwellEq}, \eqref{D-t-avg} only if $\omega_I=0$. In fact, for $\omega_I\neq 0$, the averaged model \eqref{D-t-avg} is not defined. In \eqref{MaxwellEq}, \eqref{Dispalacement} the linear terms are proportional to $e^{\omega_I t}$ while the nonlinear terms are proportional to $e^{3\omega_I t}$. This is clearly in contrast with the linear problem ($\hat{\chi}^{(3)}=0$), where complex frequencies are allowed. In fact, the classical linear surface plasmons on metallic surfaces have frequencies with $\Imag(\omega)< 0$ resulting in an exponential decay in time.
    \end{rem}

	In our analysis we will first prove bifurcations of solutions of \eqref{system}  with complex $\omega$. Second, we show that under the so called $\PT$-symmetry assumption on the functions $\hat{\chi}^{(1,3)}(\cdot,\omega)$ bifurcations with real $\omega$ can be proved.  The $\PT$-symmetry results in a spatial balance of loss and gain of the material.

	The concrete form of the functions  $\hat{\chi}^{(1,3)}$ will not be important for our results. These material functions are given by measurements for the given material. A classical description of (homogeneous)  metals is via the Drude model \cite{Ashcroft76}
	\begin{equation}\label{Drude}
		\hat{\chi}^{(1)}(x,\omega)= -\frac{\omega_p^2}{\omega^2+\ri\gamma\omega},
	\end{equation}
	where $\omega_p\in\R^+$ and $\gamma\in\R$. In dielectrics, on the other hand, one often uses an $\omega$-independent approximation of $\hat{\chi}^{(1)}$. Such an approximation is valid in a certain range of operating frequencies. In homogenous dielectrics one then  has $\hat{\chi}^{(1)}=\text{const.}\in \R.$ The realness of $\hat{\chi}^{(1)}$ means that the material is conservative, i.e no energy loss or gain is present. Also $\hat{\chi}^{(3)}$ typically takes the form of a rational function, see Chapter 1 in \cite{boyd}, or it is again assumed to be a real constant (for homogenous materials).

SPPs have been observed and studied since the 1950s, see e.g \cite{raether2006surface},  \cite{ritchie1973surface} and \cite{Ritchie-1957}. Classical SPPs arise from the interaction between an illuminating wave and the free electrons of a conductor and generate a highly confined electromagnetic field at the interface of a metal and a dielectric. This makes them useful in sensing (see \cite{anker2008biosensing}, \cite{homola1999surface} or for example \cite{guo2015strategies}). The strong localization is useful also in applications requiring light propagation in sub-wavelength geometries \cite{barnes2003surface}. In addition, the strong field near the interface enhances the nonlinear response of the medium, see \cite{Smolyaninov-2005}, \cite{Valev2012}, and SPPs propagating in the form of solitons or solitary waves have been studied \cite{Davoyan:09}, \cite{Feigenbaum:07}. These and similar studies give a numerical or a formal analytical evidence of such solutions of Maxwell's equations. We give an analytical proof of the existence of localized time harmonic SPPs bifurcating for their linear counterparts. In the case of the  transverse electric (TE) polarized fields this bifurcation analysis was carried out in \cite{dohnal2021eigenvalue,DR-corr-22}. Here we address the TM-polarization. Note that in the simplest setting of a single interface between a linear homogenous metal and a linear homogenous dielectric, only TM-SPPs exist.

	Interfaces of materials can be understood as waveguides. In fact, a classical waveguide is given by a material sandwiched between two other materials. Such a setting with dielectric homogeneous materials was studied in the Maxwell's equations both for the TE and TM polarization with the finite material being Kerr nonlinear in \cite{tikhov2022theory}, \cite{valovik2016novel} and \cite{valovik2008propagation}.
	Nonlinear dispersion relations were derived in these references. In our case the layers need not be homogenous, all layers can be nonlinear and we allow $\omega$-dependence of the material constants. In addition, the number of layers is arbitrary in our bifurcation result.

	As explained in Remark \ref{rem:cplx-om}, in our time-independent nonlinear equation only solutions with real frequencies $\omega$ produce solutions of the nonlinear Maxwell's equations. We achieve the realness of $\omega$ by working in $\PT$-symmetric materials.  $\PT$-symmetry was originally proposed in quantum mechanics in \cite{Bender-Boettcher-1998}. In the context of SPPs it has been used in  \cite{AD14}, \cite{Barton2018}, \cite{dohnal2021eigenvalue,DR-corr-22},  \cite{GP-2006}, and  \cite{Han_2014}. Mathematically, we reduce the bifurcation problem to one in a $\PT$-symmetric subspace. Such approach was used also in \cite{DS2016} in a general bifurcation problem but with a linear dependence on the bifurcation parameter. In our case the bifurcation parameter is $\omega$ and the dependence on it in the SPP case is generally nonlinear.

	\subsection{Main Results}

We consider interfaces of two or more media. These are generally inhomogeneous but with a smooth dependence of $\hat{\chi}^{(1,3)}(x,\omega)$ on $x$ within each layer. In the case of $m$ material layers we write
$$\R=\overline{\cup_{j=1}^m I_j}, \ I_j:=(x_{j-1},x_{j}), $$
where $x_0=-\infty, x_m=\infty$, and $x_{j-1} <x_j$ for all $j=1,\dots,m$.	We use the following notation:
	$$\langle \cdot,\cdot\rangle:=\langle\cdot,\cdot \rangle_{L^2(\R,\C^n)}, \quad \|\cdot\|:=\|\cdot\|_{L^2(\R)},$$
	$$\mathcal{H}^1:=\{f\in L^2(\R,\C^n): f|_{I_j}\in H^1(I_j,\C^n) \ \forall j =1,\dots,m\}, \ n\in \N,$$
	$$\|f\|_{\cH^1}:=\sum_{j=1}^m\|f\|_{H^1(I_j)},$$
	where the value of $n$ depends on the context. The cases $n=1,2,3$ appear in the paper.

	Let us now fix the functional analytic setting. Working in the Hilbert space $L^2(\R, \C^3)$, the domain of the operator $A$ is naturally chosen as
	\begin{equation}\label{domain}
		\begin{split}
			D(A):=\left\{u\in L^2\left(\mathbb{R},\mathbb{C}^3\right):\quad \nabla_k\times\ u\in L^2\left(\mathbb{R}, \mathbb{C}^3\right)\right\}.
		\end{split}
	\end{equation}
	Assuming that $\hat{\chi}_1(\cdot,\omega)\in L^\infty(\R)$, we clearly have that $B(\cdot,\omega):D(A)\to L^2(\R, \C^3)$ is bounded and $D(L_k(\omega))=D(A)$ for every $\omega$ in the domain of $V(x,\cdot)$. The range of $L_k(\omega)$ is $L^2(\R, \C^3)$ for every $\omega$ in the domain of $V(x,\cdot)$. Note that $k$ is fixed throughout the paper and the operator $L_k$ is considered as a pencil with respect to the parameter $\omega$.

	One can easily see that
	\begin{equation}\label{domain2}
		D(A)=\left\{u\in L^2\left(\mathbb{R},\mathbb{C}^3\right):\quad u_2,\ u_3\in H^1\left(\mathbb{R}, \mathbb{C}\right)\right\}.\end{equation}
For the linear problem, it is useful to study the operator $L_k(x,\omega)$ in each material layer separately. Clearly, we have
	\begin{equation}\label{domain3}
		\begin{aligned}
			D(A)=&\left\{u\in L^2\left(\R,\mathbb{C}^3\right):\  u_2,\ u_3\in H^1\left(I_j, \mathbb{C}\right) \ \forall j\in \{1,\dots,m\} \text{  and } \eqref{IFCs} \text{ holds}\right\}.
		\end{aligned}
	\end{equation}
	\begin{equation}\label{IFCs}
		\llbracket u_2 \rrbracket =\llbracket u_3\rrbracket=0 \ \text{  at  } x_j \ \forall j\in \{1,\dots,m-1\},
	\end{equation}
	where
	\[\llbracket{f}\rrbracket=0\ \text{  at  } x_j \text{  means  } \lim_{x\to {x_j}^+}f\left(x\right)=\lim_{x\to x_j^-}f(x).\]
	We assume $V\in L^\infty(\R,\C)$ for all $\omega \in \Omega \subset \C$ and define
	$$V_j(\cdot, \omega):=V(\cdot,\omega)|_{I_j}, \quad j\in \{1,\dots,m\}$$
	for each $\omega \in \Omega.$
		Defining $\Omega_j:=\{\omega \in \C: V_j(\cdot,\omega)\in L^\infty(I_j))\}$, we have

	$$\Omega =\cap_{j=1}^m\Omega_j.$$
	Note that typically $\Omega \neq \C$ as $\hat{\chi}_1$ usually has poles in the $\omega$-variable, see e.g. the Drude model.

		\noindent The following assumptions are used in Theorems \ref{mainthm_loc_bifurc} and \ref{T:PT_sym}.
	\begin{enumerate}
		\item[(A-E)] $\omega_0\neq 0$ is an algebraically simple, isolated eigenvalue of the operator pencil $L_k$;
		\item[(A-T)] $\langle\pa_\omega B(\cdot,\omega_0)\varphi_0, \varphi_0^*\rangle\neq 0$.
	\end{enumerate}
	There exists $\delta>0$ such that
	\begin{enumerate}
		\item[(A-V)]  $V_j(x,\cdot):\C\to \C$ is holomorphic on $B_\delta(\omega_0)\subset \C$ for each $j\in \{1,\dots,m\}$ and almost every $x\in I_j$ and
		$$V_j(\cdot,\omega), \frac{1}{V_j(\cdot,\omega)}, \pa_\omega V_j(\cdot,\omega), \pa^2_\omega V_j(\cdot,\omega) \in W^{1,\infty}(I_j,\C) \quad \forall j\in\{1,\dots,m\}, \omega\in B_\delta(\omega_0);$$
		\item[(A-Na)] $\C \ni \omega \mapsto\hat{\chi}^{(3)}(\cdot,\omega)\in W^{1,\infty}(I_j,\C)$ is Lipschitz continuous on $B_\delta(\omega_0)$ for each $j\in \{1,\dots,m\}$, i.e. there is $L_a>0$ such that $$\max_{j\in\{1,\dots,m\}}\|\hat{\chi}^{(3)}(\cdot,\omega_1)-\hat{\chi}^{(3)}(\cdot,\omega_2)\|_{W^{1,\infty}(I_j)}\leq L_a |\omega_1-\omega_2|$$ for all $\omega_1,\omega_2\in B_\delta(\omega_0)$;
	\end{enumerate}

	\bigskip
	\bigskip

	\begin{thm}\label{mainthm_loc_bifurc}
		Let $k \in\R$. Assume (A-E), (A-V), (A-T), and (A-Na). Let $\varphi_0 \in D(A)$ be an eigenfunction corresponding to $\omega_0$ normalized to $\|\varphi_0\|=1$ and  $\varphi_0^*$ the eigenfunction of the adjoint $L_k^*$ with the eigenvalue $\overline{\omega_0}$ normalized to $\langle\varphi_0, \varphi_0^*\rangle =1.$

		Then there is a unique branch of solutions $(\omega,u) \in \C\times (D(A)\cap \mathcal{H}^1)$ of \eqref{system}
		bifurcating from $(\omega_0,0)$. There exists $\varepsilon_0>0$ s.t. for any $\varepsilon\in(0,\varepsilon_0)$ the solution $(\omega,u)$ with $\langle u, \varphi_0^*\rangle = \eps^{1/2}$ has the form
		\begin{equation}\label{E:om-u-exp-A}
			\omega=\omega_0+\varepsilon\nu+\varepsilon^{2}\sigma,\qquad\quad u=\varepsilon^\frac{1}{2}\varphi_0+\varepsilon^\frac{3}{2}\phi+\varepsilon^\frac{5}{2}\psi,
		\end{equation}
		where
		\beq\label{E:nuA}
		\nu = -\frac{\langle h(\cdot,\omega_0,\varphi_0),\varphi_0^*\rangle}{\langle\pa_\omega B(\cdot,\omega_0)\varphi_0, \varphi_0^*\rangle},
		\eeq
		$\phi$ is a unique solution  in $D(A)\cap \cH^1\cap \langle\varphi_0^*\rangle^\perp$ of
		\beq	\label{E:phi-eq-A}
		L_k(x,\omega_0)\phi = h(x,\omega_0,\varphi_0)+\nu \pa_\omega B(x,\omega_0)\varphi_0,
		\eeq
		$\sigma \in\C$ and $\psi \in D(A)\cap \cH^1\cap \langle\varphi_0^*\rangle^\perp.$
	\end{thm}

	\begin{thm}\label{T:PT_sym}
		In the setting of Theorem \ref{mainthm_loc_bifurc} assume in addition $\omega_0\in\R$ and the $\PT$-symmetry of the material, i.e. $V(x,\omega)=\overline{V(-x,\omega)}$ and $\hat{\chi}^{(3)}(x,\omega)=\overline{\hat{\chi}^{(3)}(-x,\omega)}$ for all $x\in \R$ and all $\omega \in (\omega_0-\delta,\omega_0+\delta)$ with some $\delta>0$. Then the bifurcating solution family with $\eps\in (0,\eps_0)$ satisfies $\omega \in \R$ and $u$ can be chosen $\PT$ symmetric, i.e. $u(x)=\overline{u(-x)}$.
	\end{thm}

	\bigskip

		Our construction of a family of solutions bifurcating from an eigenvalue $\omega_0$ uses the  Fredholm property of $L_k(\omega_0)$ -  in particular the closedness of its range - as well as the algebraic simpleness of $\omega_0$. As $L_k$ is an operator pencil, some care has to be given to defining the spectrum as well as simpleness and isolatedness of eigenvalues. We proceed analogously to \cite{brown2022spectrum}, where the second order formulation corresponding to \eqref{system} was studied. Note that unlike in \cite{brown2022spectrum} we have not included the linear divergence condition $\nabla_k\cdot \left((1+\hat{\chi}^{(1)}(\cdot,\omega))u\right)=0$ in the domain $D(A)$. In \cite{brown2022spectrum} this was included to obtain $\nabla \cdot D=0$ also if $\omega=0$. In this paper we are interested only in $\omega$ near some eigenvalue $\omega_0\neq 0$. In fact, including the divergence condition in the definition of $D(A)$  would complicate the nonlinear analysis as the range of $L_k(\omega)$ would include only divergence free functions. The nonlinearity $h(\cdot,\omega,u)$ is, however, generally not divergence ($\nabla_k\cdot$) free for $u$ with  $\nabla_k\cdot \left((1+\hat{\chi}^{(1)}(\cdot,\omega))u\right)=0$.

	In applications $\hat{\chi}_1$ is often complex valued. As a result $L_k(\omega)$ is not self-adjoint and the existence of a real linear eigenvalue $\omega_0$ (or a real nonlinear eigenvalue $\omega$) cannot be expected. However, we show that for $\PT$ symmetric metamaterials (i.e. with a spatial balance of gain and loss) real linear eigenvalues can be obtained. These persist to real nonlinear eigenvalues. In this way nonlinear transverse magnetic surface plasmons with real frequencies are found. Such surface plasmons are conservative.

	The structure of the rest of the paper is as follows. In Section \ref{Section_Lep} we first define the spectrum for general operator pencils. Next, we derive explicit conditions for the existence of eigenvalues in the cases of two and three homogeneous layers. We also study their simplicity and isolatedness from the rest of the spectrum. A numerical example is provided where eigenvalues are computed and tested for simplicity and isolatedness. Theorems \ref{mainthm_loc_bifurc} and \ref{T:PT_sym} are proved in Sections \ref{bifurc} and \ref{S:PT-sym} respectively.
In Section \ref{S:numerics} we present a finite difference numerical method for the computation of the bifurcating solutions. Numerical results are shown to confirm the asymptotics given by Theorem \ref{mainthm_loc_bifurc} and Theorem \ref{T:PT_sym}. Finally, the two appendices provide some supporting calculations for the spectral analysis and for the bifurcation proof.

	\section{Linear spectral problem}\label{Section_Lep}

	Similarly to \cite{brown2022spectrum} we define the spectrum of the operator pencil $L_k$ using an additional parameter $\lambda$. Specifically, one considers the standard eigenvalue problem
	\begin{equation}\label{Lpb}
		L_k(\omega)u=\lambda u.
	\end{equation}
	Whether $\omega$ belongs to the spectrum (or its subset) of $L_k$ is defined below by the condition that $\lambda=0$ belongs to the corresponding set for $L_k(\omega)$ with $\omega$ fixed.

	We first define the \textbf{resolvent set}  of  the pencil $L_k$  by
	$$\rho(L_k):=\{\omega\in \Omega: L_k(\omega):D(A)\to L^2(\R, \C^3) \text{ is bijective with a bounded inverse}\}$$
	and  the \textbf{spectrum} of $L_k$  by
	\beq\label{E:Pspec}
	\sigma(L_k):=\Omega\setminus \rho(L_k).
	\eeq
	Note that $\sigma(L_k)=\{\omega\in \C: 0\in \sigma(L_k(\omega))\}$.

	The \textbf{point spectrum} is defined by
	$$\sigma_p(L_k):=\{\omega\in \Omega: \exists u \in D(A)\setminus \{0\}: L_k(\omega)u=0 \}.$$
	Elements of $\sigma_p(L_k)$ are called eigenvalues of $L_k$.

	The \textbf{discrete spectrum} is defined via
	\beq\label{E:sp-def-1D}
	\omega \in \sigma_d(L_k)  \ :\Leftrightarrow \ 0\in \sigma_d(L_k(\omega)), \omega \in \Omega,
	\eeq
	i.e. $\lambda=0$ is an isolated eigenvalue of finite algebraic multiplicity of the standard eigenvalue problem \eqref{Lpb} (with $\omega\in \Omega$ fixed).

	Here note that the algebraic multiplicity of $\lambda$ as an eigenvalue of $L_k(\omega)$ is called infinite if its geometric multiplicity, i.e.~$\dim\ker(L_k(\omega))$, is infinite or there exists a sequence $(u_n)_{n\in\N_0}$ of linearly independent elements $u_n\in D(A)$ such that $(L_k(\omega))u_{n+1}= u_n$ for all numbers $n\in \N_0$ with the function  $u_0\in \ker(L_k(\omega))\setminus \{0\}$. Otherwise the algebraic multiplicity is called finite.

	Finally $\omega \in \sigma_p(L_k)$ is called algebraically simple if  $\lambda=0$ is an algebraically simple eigenvalue of $L_k(\omega)$, i.e.
	if it is geometrically simple and there is no solution $u\in D(A)$ of
	\beq\label{E:gen-evec-eq}
	L_k(\omega)u= v,
	\eeq
	where $v\in \ker(L_k(\omega))\setminus \{0\}$ and such that $u$ and $v$ are linearly independent.

	In order for the spectral theory to be meaningful, we need $A$ (and hence also $L_k$) to be a closed and densely defined operator.
	We shall prove the closedness and the denseness results for the operator $L_k(\omega)$.
	\begin{prop}\label{closedness}
		The operators $A:\ D(A)\longrightarrow L^2\left(\mathbb{R},\mathbb{C}^3\right)$ and  $L_k(\omega):\ D(A)\longrightarrow L^2\left(\mathbb{R},\mathbb{C}^3\right)$ are closed and densely defined.
	\end{prop}
	\begin{proof}
		The denseness of $D(A)$ in $L^2(\R,\C^3)$ is obvious. We first show that $A: D(A)\to L^2(\R, \C^3)$ is closed. For this we recall that $D(A)$ equipped with the inner product
		\begin{equation}\label{innerproductnorms}
			\langle u,\ v \rangle_{A}:=\langle u,\ v \rangle_{L^2(\mathbb{R},\mathbb{C}^3)}+\langle \nabla_k\times u,\ \nabla_k\times v \rangle_{L^2(\mathbb{R},\mathbb{C}^3)}
		\end{equation}
		is a Hilbert space, which is shown completely analogously to Lemma 3.1 in  \cite{brown2022spectrum}. We note that the norm generated by $\langle \cdot, \cdot\rangle_{A}$ is the graph norm $$\|v\|_A=\left(\|v\|_{L^2(\R, \C^3)}^2 + \|\nabla_k \times v\|_{L^2(\R, \C^3)}^2 \right)^{1/2}$$
		and $A:D(A)\to L^2(\R,\C^3)$ is bounded if $D(A)$ is equipped with $\|\cdot\|_A$.

		Let now $(u_k)\subset D(A)$ and $(u_k,Au_k)\to (u,v)$ in $L^2(\R,\C^3)^2$. This implies that $(u_k)$ is a Cauchy sequence in $(D(A),\|\cdot\|_A)$ and hence $u\in D(A)$. By the continuity of $A$ we then get $Au_k\to Au$ in $L^2$ and hence $v=Au$.

		Since $V(\cdot,\omega)\in L^\infty(\R)$, we get that $B(\cdot,\omega):  D(A) \to L^2(\R, \C^3)$ is bounded, and conclude that the operator $L_k(\cdot,\omega)$ is a closed and densely defined operator in $L^2(\R, \C^3)$.
	\end{proof}

	\blem\label{L:phi0-reg}
	Let $\omega_0$ be an eigenvalue of $L_k$ and assume $\frac{1}{V_j(\cdot,\omega_0)}\in W^{1,\infty}(I_j,\C)$ for all $j\in\{1,\dots,m\}$. Then corresponding  eigenfunctions $\varphi_0\in D(A)\setminus \{0\}$ satisfy $\varphi_0\in \cH^1.$
	\elem
	\bpf
	We have $\varphi_{0,2}, \varphi_{0,3}\in H^1(I_j)$ for each $j$ by the definition of $D(A)$. The first component is given by
	$\varphi_{0,1}=\frac{k\varphi_{0,3}}{V(\cdot,\omega_0)}$, which is in $H^1(I_j)$ for each $j$ as $\frac{1}{V_j(\cdot,\omega_0)}\in W^{1,\infty}(I_j) $.
	\epf

\subsection{Homogeneous layers}

In the case of homogeneous layers, i.e.
$$V_j(\cdot,\omega)=V_j(\omega)\in \C, \quad j \in \{1,\dots,m\},$$
the fundamental system  for the linear part of \eqref{system} can be found explicitly and the condition for $\omega$ to be an eigenvalue reduces effectively to an algebraic equation.

We study eigenvalues in the case of homogeneous layers only outside the set
	$$\Omega_0:=\{\omega\in \Omega: \omega=0 \ \text{ or } \  V_j(\omega)= 0 \text{  for some } j \in \{1,\dots,m\}\}.$$
This set consists of eigenvalues of infinite multiplicity. Assuming namely $V_j(\omega)=0$, then $L_k\nabla_k \varphi=\nabla_k \times (\nabla_k \varphi)=0$ for each $\varphi\in C^\infty_c(I_j,\R)$ (curl of a gradient vanishes and $B(x,\omega)(v,w,0)^T=0$ for each $x\in I_j, v,w\in L^2(\R)$).

	In contrast, in \cite{brown2022spectrum}, where $D(A)$ includes the divergence condition, $\omega\in \Omega_0$ is an eigenvalue of infinite multiplicity only if $1+\hat{\chi}^{(1)}(\cdot,\omega)=0$ on one of the layers.

Next, we consider two special cases, namely $m=2 $ and $m=3$.

	\subsubsection{Two homogenous layers}\label{S:2layers}
	Without loss of generality we choose the two homogeneous layers $I_1:=\R_-:=(-\infty,0), I_2:=\R_+:=(0,\infty)$ (i.e. with the interface at $x=0$). Hence
	$$V(x,\omega) = V_\pm(\omega) \quad \text{for } \pm x>0,$$
	where $V_\pm(\omega):=V(\pm x>0, \omega)$ are independent of $x$. We also define the functions
	$$W(x, \omega):=-\omega V(x, \omega) \ \text{ and } \ W_\pm(\omega):=-\omega V_\pm(\omega).$$

	The spectrum of the problem with two layers was analysed in \cite{brown2022spectrum}  in detail for the second order formulation (the curl curl problem) with the condition $\nabla_k\cdot \left((1+\hat{\chi}^{(1)}(\cdot,\omega))u\right)=0$ included in the definition of the domain of $A$. As explained in Section \ref{S:intro}, this definition of $D(A)$ is not suitable for the nonlinear analysis. Our domain $D(A)$ excludes this condition and hence the range of $L_k$ is not divergence free. As a result, in the second order formulation the resolvent problem cannot be reduced to a scalar equation since derivatives of $L^2$ functions would appear, see the proof of Proposition 3.2 in \cite{brown2022spectrum}. Hence, there is no substantial benefit in using the second order formulation and we stay within the first order formulation. For this we need to redo some of the calculations in \cite{brown2022spectrum} with the definition of $D(A)$ as in \eqref{domain2}.

	For the resolvent set we have
	\begin{prop}\label{T:resolv-1D}
		Let $k\in \R$. Then
		$$\rho(L_k)\supset \Omega \setminus ( M^{(k)}_+\cup M^{(k)}_- \cup N^{(k)}\cup \Omega_0),$$
		where
		$$M^{(k)}_\pm:= \{ \omega \in \Omega\setminus\Omega_0: W_\pm (\omega)\in [k^2,\infty)\},$$
		and
		\beq
		N^{(k)}:= ~\{ \omega \in  \Omega \setminus \Omega_0: W_+(\omega),W_-(\omega)\notin [k^2,\infty) \text{ and } \eqref{E:ev.cond-1D} \text{ holds}  \}. \label{E:Nk}
		\eeq
		\beq\label{E:ev.cond-1D}
		\sqrt{k^2-W_+(\omega)} V_-(\omega)+ \sqrt{k^2-W_-(\omega)} V_+(\omega)=0.
		\eeq
	\end{prop}

	\begin{proof}
		Let $r \in L^2(\R, \C^3)$ and $\omega \in  \Omega \setminus ( M^{(k)}_+\cup M^{(k)}_- \cup N^{(k)}\cup \Omega_0)$. We need to show the existence of a unique $u\in D(A)$ such that
		\beq\label{Neq_22}
		\left. \begin{array}{rcl}
			\ri ku_3 - \ri  V_\pm(\omega)u_1&=& r_1 \\
			- u_3 ' -\ri  V_\pm(\omega)u_2  &=& r_2 \\
			u_2 '  -\ri k u_1 - \ri \omega  u_3     &=& r_3
		\end{array} \right\} {\rm ~on~} \R_\pm,
		\eeq
		where $V(\omega):=V(x,\omega)$ and such that $\|u\|_{L^2(\R)^3}\leq c \|r\|_{L^2(\R)^3}$ with $c$ independent of $r$. The first equation in \eqref{Neq_22} implies
		\begin{equation}\label{u_1}u_1=\frac{\ri}{ V_\pm(\omega)}\left(r_1-\ri k u_3\right) \  {\rm ~on~} \R_\pm.\end{equation}
		Plugging \eqref{u_1} into the third equation of \eqref{Neq_22}, one obtains
		\beq\label{u23}
		\left. \begin{array}{rcl}
			u_3 ' +\ri  V_\pm(\omega)u_2  &=& -r_2 \\
			u_2 '  + \ri  \frac{W_\pm(\omega)-k^2}{ V_\pm(\omega)} u_3     &=& r_3-\frac{k}{ V_\pm(\omega)}r_1
		\end{array} \right\} {\rm ~on~} \R_\pm.
		\eeq
		The homogeneous version of system \eqref{u23} (with the solution vector being $(u_2,u_3)^T$) has the fundamental solution matrix
		\beq\label{Fund_Matrix}
		Y_\pm(x):=\begin{pmatrix}
			\mu_\pm e^{\mu_\pm x} \ & \mu_\pm e^{-\mu_\pm x}\\
			-\ri V_\pm(\omega)e^{\mu_\pm x} \ & \ri V_\pm(\omega)e^{-\mu_\pm x}
		\end{pmatrix},
		\eeq
		where $\mu_\pm :=\sqrt{k^2-W_\pm(\omega)}$. Analogously to  \cite{brown2022spectrum} (see Lemma 3.5), we have the following particular solutions in $L^2(\R_\pm,\C^2)$
		\begin{align*}
			\tilde{u}_p^+(x)&:=\frac{1}{2}Y_+(x)\begin{pmatrix}
				\int_{x}^{\infty}\rho_+^{(1)}(s)e^{-\mu_+s}\dd s\\
				\int_{0}^{x}\rho_+^{(2)}(s)e^{\mu_+s}\dd s
			\end{pmatrix}\\
			&= \frac{1}{2}\begin{pmatrix}
				e^{\mu_+x}\mu_+\int_{x}^{\infty}\rho_+^{(1)}(s)e^{-\mu_+s}\dd s+e^{-\mu_+x}\mu_+\int_{0}^{x}\rho_+^{(2)}(s)e^{\mu_+s}\dd s\\
				e^{\mu_+x}(-\ri V_+(\omega))	\int_{x}^{\infty}\rho_+^{(1)}(s)e^{-\mu_+s}\dd s+e^{-\mu_+x}\ri V_+(\omega)\int_{0}^{x}\rho_+^{(2)}(s)e^{\mu_+s}\dd s
			\end{pmatrix}, \quad x\in \R_+,
		\end{align*}
		\begin{align*}
			\tilde{u}_p^-(x)&:=\frac{1}{2}Y_-(x)\begin{pmatrix}
				\int_{x}^{0}\rho_-^{(1)}(s)e^{-\mu_-s}\dd s\\
				\int_{-\infty}^{x}\rho_-^{(2)}(s)e^{\mu_-s}\dd s
			\end{pmatrix}\\
			&= \frac{1}{2}\begin{pmatrix}
				e^{\mu_-x}\mu_-\int_{x}^{0}\rho_-^{(1)}(s)e^{-\mu_-s}\dd s+e^{-\mu_-x}\mu_-\int_{-\infty}^{x}\rho_-^{(2)}(s)e^{\mu_-s}\dd s\\
				e^{\mu_-x}(-\ri V_-(\omega))	\int_{x}^{0}\rho_-^{(1)}(s)e^{-\mu_-s}\dd s+e^{-\mu_-x}\ri V_-(\omega)\int_{-\infty}^{x}\rho_-^{(2)}(s)e^{\mu_-s}\dd s
			\end{pmatrix}, \quad x\in \R_-,
		\end{align*}
		where
		\beq\label{E:rhos}
		\rho_\pm^{(1)}(x):=\ri\frac{r_2(x)}{V_\pm(\omega)}+\frac{kr_1(x)}{V_\pm(\omega)\mu_\pm}-\frac{r_3(x)}{\mu_\pm},\quad \rho_\pm^{(2)}(x):=\ri\frac{r_2(x)}{V_\pm(\omega)}-\frac{kr_1(x)}{V_\pm(\omega)\mu_\pm}+\frac{r_3(x)}{\mu_\pm}.
		\eeq
		Because of $\Real(\mu_\pm)>0$ (recall that  $\omega \notin M_\pm^{(k)}$) it holds that $\tilde{u}_p^+\in L^2(\R_+, \C^2)$,  $\tilde{u}_p^-\in L^2(\R_-, \C^2)$.

		The general solution $\tilde{u}=(u_2, u_3)\in L^2(\R,\C^2)$ of \eqref{u23} is given by
		\beq\label{general_solu}
		\tilde{u}(x)=
		\begin{cases}
			C_-e^{\mu_- x}\begin{pmatrix}
				\mu_- \\ -\ri V_-(\omega)
			\end{pmatrix} + \tilde{u}_p^-(x), & x<0, \\
			C_+e^{-\mu_+ x}\begin{pmatrix}
				\mu_+ \\ \ri V_+(\omega)
			\end{pmatrix} + \tilde{u}_p^+(x), & x>0
		\end{cases}
		\eeq
		with $C_\pm \in \C$ arbitrary. For $\omega \notin \Omega_0$ we have $V_\pm(\omega)\neq 0$ and \eqref{u_1} yields also $u_1\in L^2(\R).$

		For the interface condition $\llbracket u_2\rrbracket=0$ (at $x=0$) first note that
		\beq\label{IFC_u_2}\displaystyle\lim_{x\to 0^+}u_2(x)=C_+\mu_++\frac{\mu_+}{2}\int_{0}^{\infty}\rho_+^{(1)}(s)e^{-\mu_+s}\dd s,\quad \lim_{x\to 0^-}u_2(x)=	C_-\mu_-+\frac{\mu_-}{2}\int_{-\infty}^{0}\rho_-^{(2)}(s)e^{\mu_-s}\dd s.
		\eeq
		Hence $\llbracket u_2\rrbracket=0$ if and only if
		\begin{equation}\label{C_plus}
			C_+=\frac{\mu_-}{\mu_+}C_-+\frac{\mu_-}{2\mu_+}\int_{-\infty}^{0}\rho_-^{(2)}(s)e^{\mu_-s}\dd s-\frac{1}{2}\int_{0}^{\infty}\rho_+^{(1)}(s)e^{-\mu_+s}\dd s.
		\end{equation}

		For $\llbracket u_3\rrbracket=0$ we note that
		\beq\label{IFC_u_3}\begin{split}
			\displaystyle\lim_{x\to 0^+}u_3(x)&=C_+\ri V_+(\omega)-\frac{\ri V_+(\omega)}{2}\int_{0}^{\infty}\rho_+^{(1)}(s)e^{-\mu_+s}\dd s,\\
			\lim_{x\to 0^-}u_3(x)&=	-C_-\ri V_-(\omega)+\frac{\ri V_-(\omega)}{2}\int_{-\infty}^{0}\rho_-^{(2)}(s)e^{\mu_-s}\dd s.
		\end{split}
		\eeq
		Using \eqref{C_plus}, we get
		\beq\label{dis_relation_0}
		\left(\mu_-V_+(\omega)+\mu_+V_-(\omega)\right)C_-=\frac{1}{2}\left(\mu_+V_-(\omega)-\mu_-V_+(\omega)\right)\int_{-\infty}^{0}\rho_-^{(2)}(s)e^{\mu_-s}\dd s+\mu_+V_+(\omega)\int_{0}^{\infty}\rho_+^{(1)}(s)e^{-\mu_+s}\dd s. \eeq
		Hence, a unique $C_-$ exists for any $r\in L^2(\R,\C^3)$ if and only if
		\[\mu_-V_+(\omega)+\mu_+V_-(\omega)\neq 0,\]
		i.e. $\omega\notin N^{(k)}$.

		To prove $\|u\|_{L^2(\R,\C^3)}\leq c \|r\|_{L^2(\R,\C^3)}$ with $c$ independent of $r$, note that  Lemma 3.5 of \cite{brown2022spectrum} provides an estimate for all the integral terms in  $\tilde{u}_p^\pm$:
		\beq\label{E:est-rho-int}
		\begin{aligned}
			\left\|e^{\mu_+ \cdot }\int_{\cdot}^{\infty}\rho_+^{(1)}(s)e^{-\mu_+s}\dd s\right\|_{L^2(\R_+)}&\leq \frac{c}{\Real(\mu_+)}\left\|\rho_+^{(1)}\right\|_{L^2(\R_+)}, \\
			\left\|e^{-\mu_+ \cdot }\int_{0}^{\cdot}\rho_+^{(2)}(s)e^{\mu_+s}\dd s\right\|_{L^2(\R_+)}&\leq \frac{c}{\Real(\mu_+)}\left\|\rho_+^{(2)}\right\|_{L^2(\R_+)}, \\
			\left\|e^{-\mu_- \cdot }\int_{-\infty}^{\cdot}\rho_-^{(2)}(s)e^{\mu_-s}\dd s\right\|_{L^2(\R_-)}&\leq \frac{c}{\Real(\mu_-)}\left\|\rho_-^{(2)}\right\|_{L^2(\R_-)}, \\
			\left\|e^{\mu_- \cdot }\int_{\cdot}^{0}\rho_-^{(1)}(s)e^{-\mu_-s}\dd s\right\|_{L^2(\R_-)}&\leq \frac{c}{\Real(\mu_-)}\left\|\rho_-^{(1)}\right\|_{L^2(\R_-)}.
		\end{aligned}
		\eeq
		For $\omega \in \Omega \setminus  ( M^{(k)}_+\cup M^{(k)}_- \cup \Omega_0)$ we clearly have $$\|\rho_\pm^{(1)}\|_{L^2(\R_\pm)}\leq c \|r\|_{L^2(\R_\pm,\C^3)}\  \text{ and } \ \|\rho_\pm^{(2)}\|_{L^2(\R_\pm)}\leq c \|r\|_{L^2(\R_\pm,\C^3)}.$$
		In summary, we get $\|\tilde{u}\|_{L^2(\R,\C^2)}\leq c \|r\|_{L^2(\R,\C^3)}$. For $u_1$ we get from \eqref{u_1} the estimate \[\|u_1\|_{L^2(\R)}\leq c \left(\|r_1\|_{L^2(\R)}+\|u_3\|_{L^2(\R)}\right).\] Altogether we conclude $\|u\|_{L^2(\R,\C^3)}\leq c \|r\|_{L^2(\R,\C^3)}$.

		The only missing property for $u\in D(A)$ is $\tilde{u}\in H^1(\R_\pm,\C^2)$. This follows, however, from the estimates \eqref{E:est-rho-int} because we have, for example,
		$$
		u_2'(x)=\mu_-^2\left(C_-e^{\mu_-x} + \frac{1}{2}\left(e^{\mu_-x}\int_{x}^{0}\rho_-^{(1)}(s)e^{-\mu_-s}\dd s-e^{-\mu_-x}\int_{-\infty}^{x}\rho_-^{(2)}(s)e^{\mu_-s}\dd s\right)\right) +\frac{\mu_-}{2}( \rho_-^{(2)}(x) -\rho_-^{(1)}(x))
		$$
		for $x\in \R_-$.

		The proof of Proposition \ref{T:resolv-1D} is complete.
	\end{proof}
	\begin{rem}
		Note that Proposition \ref{T:resolv-1D} holds also with $\Omega_0$ replaced by the possibly smaller
		set $\tilde{\Omega}_0:=\{\omega\in \Omega: V_j(\omega)= 0 \text{  for some } j \in \{1,\dots,m\}\}$. In other words, $0\in \rho(L_k)$ if $0\in \Omega$, $V_j(0)\neq 0 \ \forall j$, and $k\neq 0$. The inequality $V_j(0)\neq 0$ is possible if $\hat{\chi}^{(1)}$ has a pole at $\omega=0$. If $k=0$, then $M_\pm^{(k)}=\{\omega \in \Omega \setminus \tilde{\Omega}_0: -\omega V_\pm(\omega)\in [0,\infty)\}$ and hence $0\in M_+^{(k)}\cup M_-^{(k)}$ as long as $V_+$ or $V_-$ has no pole at $0$.
	\end{rem}
	\begin{rem}
		Equation \eqref{E:ev.cond-1D} is equivalent to the dispersion relation (2.14) in \cite{maier2007plasmonics}.
	\end{rem}

	The next result determines the set of eigenvalues of $L_k$ away from the set $\Omega_0$ for any $k\in \R$, and shows their simplicity.
	\begin{prop}\label{T:pt-spec-1D}
		Let $k \in \R$. Then
		\beq\label{E:pt-sp1Dknz-omnz}
		\begin{aligned}
			\sigma_p(L_k)\setminus \Omega_0 &= N^{(k)}.
		\end{aligned}
		\eeq
		All eigenvalues in $\sigma_p(L_k)\setminus \Omega_0$ are geometrically simple. An eigenvalue $\omega \in \sigma_p(L_k)\setminus \Omega_0$ is algebraically simple if and only if
		\beq\label{E:simple-cond-2-lay}
		(2k^2-W_+(\omega))(2k^2-W_-(\omega))-2k^4-(k^2-W_+(\omega))V_-(\omega)^2\neq 0.
		\eeq
	\end{prop}

	\begin{proof}
		From the proof of Prop. \ref{T:resolv-1D} we get that $L^2(\R)$-solutions of $L_k(\omega)\psi=0$, i.e. of the  homogenous version of \eqref{Neq_22}, have the form
		\beq\label{E:psi}
		\psi(x)=\begin{cases}
			c_-e^{\mu_-x}\bspm-\ri k\\ \mu_- \\ -\ri V_-(\omega)\espm  &\quad \text{for}\ x<0, \\
			c_+e^{-\mu_+x}\bspm\ri k\\ \mu_+\\ \ri V_+(\omega) \espm  &\quad \text{for}\ x>0
		\end{cases}
		\eeq
		with $c_\pm \in \C$. For any $c_\pm \in \C$ we clearly have $\psi|_{\R_\pm}\in H^1(\R_\pm,\C^3)$. The interface conditions \eqref{IFCs} hold if and only if
		$$c_-\mu_-=c_+\mu_+ \ \text{ and } \ - c_-V_-(\omega)= c_+V_+(\omega),$$
		which is equivalent to $c_+=-\frac{V_-(\omega)}{V_+(\omega)}c_-$ and $V_+(\omega)\mu_-+V_-(\omega)\mu_+=0$, i.e. equation \eqref{E:ev.cond-1D}.

		The above unique form of the eigenfunction $\psi$ shows that each eigenvalue $\omega$ in $\sigma_p(\cL_{k})\setminus \Omega_0$ is geometrically simple in the sense that $\lambda=0$ is a geometrically simple eigenvalue of
		$$
		L_k(\omega)u = \lambda u
		$$

		Finally, we study the algebraic simpleness, i.e. the non-existence of a solution $u\in D(A)$ of \eqref{E:gen-evec-eq} with $v:=\psi$. That means, we consider
		$$L_k(\omega)u=\psi, \quad u\in D(A).$$
		Assuming for a contradiction that a solution $u$ exists, we first follow the lines of the proof of Proposition \ref{T:resolv-1D} with $r:=\psi\in L^2(\R,\C^3)$. Since
		\eqref{E:ev.cond-1D} holds, we get from \eqref{dis_relation_0}
		\begin{align}
			0&=\frac{1}{2}\left(\mu_+V_-(\omega)-\mu_-V_+(\omega)\right)\int_{-\infty}^{0}\rho_-^{(2)}(s)e^{\mu_-s}\dd s+\mu_+V_+(\omega)\int_{0}^{\infty}\rho_+^{(1)}(s)e^{-\mu_+s}\dd s, \notag\\
			&=\mu_+V_-(\omega)\int_{-\infty}^{0}\rho_-^{(2)}(s)e^{\mu_-s}\dd s+\mu_+V_+(\omega)\int_{0}^{\infty}\rho_+^{(1)}(s)e^{-\mu_+s}\dd s=:\alpha,\label{E:C2b}
		\end{align}
		with $\rho_+^{(1)}$ and $\rho_-^{(2)}$ given in \eqref{E:rhos}.

		Next, from \eqref{E:psi}, setting $c_-:=1$ (hence $c_+=-V_-(\omega)/V_+(\omega)$) and using \eqref{E:ev.cond-1D} gives
		$$\psi_1(x)=\begin{cases} 	-\ri k e^{\mu_- x},  & x<0,\\
			-\ri k \frac{V_-(\omega)}{V_+(\omega)} e^{-\mu_+ x},  & x>0,
		\end{cases} \quad
		\psi_2(x)=\begin{cases} 	\mu_- e^{\mu_- x},  & x<0,\\
			\mu_- e^{-\mu_+ x},  & x>0,
		\end{cases} \quad \psi_3(x)=\begin{cases}
			-\ri V_-(\omega) e^{\mu_- x},  & x<0,\\
			-\ri V_-(\omega) e^{-\mu_+ x},  & x>0
		\end{cases}$$
		and we obtain
		\[
		\begin{aligned}
			\int_{0}^{\infty}\rho_+^{(1)}(s)e^{-\mu_+s}\dd s=&\frac{\ri}{2\mu_+}\left(\frac{\mu_-}{V_+}-\frac{k^2V_-}{V_+^2\mu_+}+\frac{V_-}{\mu_+}\right),\\
			\int_{-\infty}^{0}\rho_-^{(2)}(s)e^{\mu_-s}\dd s=&\frac{\ri}{2\mu_-}\left(\frac{\mu_-}{V_-}+\frac{k^2}{V_-\mu_-}-\frac{V_-}{\mu_-}\right).
		\end{aligned}
		\]
		Using \eqref{E:ev.cond-1D}, i.e. $\mu_-V_++\mu_+V_-=0$, \eqref{E:C2b} implies
		$$
		\begin{aligned}
			-2\ri \frac{\alpha}{V_+}=&\frac{\mu_-}{V_+}-\frac{k^2V_-}{V_+^2\mu_+}+\frac{V_-}{\mu_+} -\frac{\mu_-}{V_-}-\frac{k^2}{V_-\mu_-}+\frac{V_-}{\mu_-}\\
			=&\mu_-\frac{V_--V_+}{V_+V_-}+k^2\frac{V_-\mu_-^2-V_+\mu_+^2}{V_+\mu_+^2V_-\mu_-}+V_-\frac{\mu_++\mu_-}{\mu_+\mu_-}\\
			=&\frac{1}{\mu_+^2\mu_-V_+V_-}\left[(V_--V_+)(\mu_-^2\mu_+^2+k^2(\mu_+^2+\mu_-^2))+k^2(\mu_-^2V_+-\mu_+^2V_-)+\mu_+V_+V_-^2(\mu_++\mu_-)\right].
			\label{contradiction}
		\end{aligned}
		$$
		Because $\mu_-^2V_+-\mu_+^2V_-=k^2(V_+-V_-)$ and $\mu_+V_+V_-^2(\mu_++\mu_-)=\mu_+^2V_-^2(V_+-V_-)$, we obtain
		$$
		\begin{aligned}
			-2\ri \frac{\alpha}{V_+} =& \frac{V_--V_+}{\mu_+^2\mu_-V_+V_-}\left[(\mu_+^2+k^2)(\mu_-^2+k^2)-2k^4-(k^2-W_+)V_-^2\right]\\
			=&  \frac{V_--V_+}{\mu_+^2\mu_-V_+V_-}\left[(2k^2-W_+)(2k^2-W_-)-2k^4-(k^2-W_+)V_-^2\right].
		\end{aligned}
		$$
		Since $V_+(\omega)\neq V_-(\omega)$ (due to \eqref{E:ev.cond-1D}), the statement follows.
	\end{proof}

	As we show next, outside the set $\Omega_0$, the eigenvalues of $L_{k}$ are isolated, i.e. the eigenvalue $\lambda=0$ of $L_k(\omega)u = \lambda u$  is isolated from the rest of the spectrum of this eigenvalue problem.

	\begin{prop}\label{T:disc_sp_1D}
		Let $k\in\R$. Every eigenvalue in $\sigma_p(L_k)\setminus \Omega_0$ is isolated.
	\end{prop}
	\begin{proof}
		Let $\omega \in \sigma_p(L_k)\setminus \Omega_0$. It is to show that $\lambda \in \rho(A-B(\omega))$ for all $\lambda \in B_r(0)\setminus \{0\}$ with $r>0$ small enough.

		\indent Analogously to Proposition \ref{T:resolv-1D} we have (replacing $B(\omega)$ by $B(\omega)+\lambda I$)
		$$\rho(L_k)\supset ~\{ \omega \in  \Omega \setminus \Omega_0: (\omega-\ri \lambda)(V_\pm(\omega)-\ri\lambda)\notin [k^2,\infty) \text{ and } \mu_-^{(\lambda)}(V_+(\omega)-\ri\lambda)+\mu_+^{(\lambda)}(V_-(\omega)-\ri\lambda)\neq 0\},$$
		where
		$$\mu_\pm^{(\lambda)}:=\sqrt{k^2+(\omega-\ri\lambda)(V_\pm(\omega)-\ri\lambda)}.$$
		Clearly, as $-\omega V_\pm(\omega)=W_\pm(\omega)\notin [k^2,\infty)$ (since $\omega \in N^{(k)}$), we get also $ -(\omega-\ri\lambda)(V_\pm(\omega)-\ri\lambda)\notin [k^2,\infty)$ for all $|\lambda|$ small enough. It remains to be shown that $\mu_-^{(\lambda)}(V_+(\omega)-\ri\lambda)\neq -\mu_+^{(\lambda)}(V_-(\omega)-\ri\lambda)$ for all $|\lambda|$ small enough.

		Assuming the equality, we get
		$$(k^2+(\omega-\ri\lambda)(V_--\ri\lambda))(V_+-\ri\lambda)^2=(k^2+(\omega-\ri\lambda)(V_+-\ri\lambda))(V_--\ri\lambda)^2,$$
		which can be simplified to
		\beq\label{E:cubic-lam}
		k^2((V_--\ri\lambda)+(V_+-\ri\lambda))=-(\omega-\ri\lambda)(V_+-\ri\lambda)(V_--\ri\lambda)
		\eeq
		after dividing by $V_+-V_-$, which is non-zero. Equation \eqref{E:cubic-lam} is a cubic equation for $\lambda \in \C$. One of the roots is $\lambda=0$ due to \eqref{E:ev.cond-1D}, which implies $k^2(V_++V_-)=-\omega V_+V_-$. We denote the other two roots by $\lambda_1,\lambda_2$. Choosing $r<\min \{|\lambda_1|,|\lambda_2|\}$, the equation does not hold for any $\lambda \in B_r(0)\setminus \{0\}$.
	\end{proof}
	\begin{rem}
	Together with Proposition \ref{T:pt-spec-1D} we conclude that
	\beq\label{E:disc-sp1Dknz-omnz}
	\sigma_d(L_k)\setminus \Omega_0=\{\omega \in N^{(k)}: \eqref{E:simple-cond-2-lay} \text{ holds}\}.
	\eeq
\end{rem}

	\subsubsection{Three homogeneous layers}\label{S:3layers}

	Next we consider three homogeneous material layers with interfaces at $x=0$ and $x=d>0$, i.e. a sandwich geometry with two unbounded layers:
	\begin{equation}\label{chi1_w_3layers_homog_const}
		V(x,\omega)=\begin{cases}
			V_-(\omega) & \text{for $x<0$,} \\
			V_*(\omega) & \text{for $x\in(0,d)$,} \\
			V_+(\omega) & \text{for $x>d$,}
		\end{cases}
	\end{equation}
	where $V_-,V_*,V_+: \Omega\subset \C\to\C$. In other words, we set $I_1:=(-\infty,0), I_2:=(0,d), I_3:=(d,\infty),$ $V_1:=V_-, V_2:=V_*,$ and $V_3:=V_+$. We also define
	$W(x, \omega):=-\omega V(x, \omega)$, $W_\pm(\omega):=-\omega V_\pm(\omega)$, and $W_*(\omega):=-\omega V_{*}(\omega)$.
	The following conditions (an effective dispersion relation) play an important role in the analysis of Section \ref{S:3layers},
	\begin{equation}\label{cond_d_cD|M|cD}
		\exists m\in \Z: \ d_m:=\frac1{2\mu_*}\left[\log\left(\frac{\left(\mu_*V_+(\omega)-\mu_+V_*(\omega)\right)\left(\mu_*V_-(\omega)-\mu_-V_*(\omega)\right)}{\left(\mu_*V_+(\omega)+\mu_+V_*(\omega)\right)\left(\mu_-V_*(\omega)+\mu_*V_-(\omega)\right)}\right)+2\pi \ri m\right]\in (0,\infty)
	\end{equation}
	and
	\beq\label{cond_d_cD|M|cD-2}
	\begin{cases}
		&\mu_*V_+(\omega)+\mu_+V_*(\omega)=0,\   \mu_*V_-(\omega)-\mu_-V_*(\omega)=0\\
		 &\text{or} \\
		 & \mu_*V_+(\omega)-\mu_+V_*(\omega)=0,\  \mu_*V_-(\omega)+\mu_-V_*(\omega)=0,
	\end{cases}
	\eeq
	where
	$$\mu_\pm:=\sqrt{k^2-W_\pm(\omega)}, \quad \text{and} \quad \mu_*:=\sqrt{k^2-W_*(\omega)}.$$

	As explained below, if $d_m\in (0,\infty)$ for some $m\in \Z$ and provided
	$$k^2-W_+(\omega) \notin (-\infty,0], k^2-W_-(\omega) \notin (-\infty,0],\ \text{and}\ k^2-W_*(\omega) \neq 0,$$
	and if \eqref{cond_d_cD|M|cD-2} does not hold, then with
	$$d:=d_m$$
	an eigenvalue of the Maxwell operator $L_k$ exists.  The equation in \eqref{cond_d_cD|M|cD} can then be reformulated as
	\beq\label{E:cond-3-layers-B}
	e^{2\mu_*d}=\frac{\left(\mu_*V_+(\omega)-\mu_+V_*(\omega)\right)\left(\mu_*V_-(\omega)-\mu_-V_*(\omega)\right)}{\left(\mu_*V_+(\omega)+\mu_+V_*(\omega)\right)\left(\mu_-V_*(\omega)+\mu_*V_-(\omega)\right)}.
	\eeq
	The term $2\pi \ri m$ in \eqref{cond_d_cD|M|cD} appears due the fact that $z=\log(b)+2\pi \ri m$ solves $e^z=b$ for any $m\in \Z$.
		\begin{rem}
	Equation \eqref{E:cond-3-layers-B} is equivalent to the dispersion relation (2.28) in \cite{maier2007plasmonics}.
\end{rem}

	Note that if all the three layers are conservative materials, i.e. $V_\pm,V_*:\Omega\to \R_+$, then \eqref{cond_d_cD|M|cD} cannot be satisfied for any real $\omega$ because $\mu_\pm, \mu_*>0$ and the absolute value of the right hand side in \eqref{E:cond-3-layers-B} is smaller than one while the left hand side is larger than one for $\mu_*,d>0$.

	For a $\PT$-symmetric example setting (metal with gain - dielectric - metal with loss) we compute $d_m$ numerically in Example \ref{Ex:3layers-eigs} showing that this setting apparently leads to the existence of linear eigenvalues $\omega_0\in \R$.

	Next, we introduce important sets for the description of the spectrum for any $k\in \R$:
	\beq\label{E:Mkpm_3layers}
	\begin{split}
		M^{(k)}_{-}:= &~\{ \omega \in  \Omega \setminus \Omega_0: W_{-}(\omega)\in [k^2,\infty)\}, \\
		M^{(k)}_{+}:= &~\{ \omega \in  \Omega \setminus \Omega_0:  W_+(\omega)\in [k^2,\infty)\},
	\end{split} \eeq
	and
	\beq
	N^{(k)}:= ~\{ \omega \in  \Omega \setminus \Omega_0: W_+(\omega), W_-(\omega)\notin [k^2,\infty), W_*(\omega)\neq k^2, \eqref{cond_d_cD|M|cD} \text{ holds and } \eqref{cond_d_cD|M|cD-2} \text{ does not hold}  \}. \label{E:Nk_3layers}
	\eeq
	\beq
	O^{(k)}:= ~\{ \omega \in  \Omega \setminus \Omega_0: W_+(\omega), W_-(\omega)\notin [k^2,\infty), W_*(\omega)\neq k^2, \text{ and } \eqref{cond_d_cD|M|cD-2} \text{ holds}  \}. \label{E:Ok_3layers}
	\eeq

	For the resolvent set we have
	\begin{prop}\label{T:resolv-1D_3layers}
		Let $k\in\R$.
		$$\rho(L_k)\supset \Omega\setminus (M^{(k)}_+\cup M^{(k)}_-\cup N^{(k)} \cup O^{(k)}\cup \Omega_0)=:\cM^{(k)}.$$
	\end{prop}

	\begin{proof}
		For the whole proof let $\omega \in \Omega \setminus \Omega_0.$	As in the proof of Proposition \ref{T:resolv-1D}, given $r\in L^2(\R,\C^3),$ we firstly need to show the existence of a unique $u\in D(A)$ such that
		\beq\label{ODEsys_3layers}
		\left. \begin{array}{rcl}
			u_3 ' +\ri  V(x,\omega)u_2  &=& -r_2 \\
			u_2 '  + \ri  \frac{W(x,\omega)-k^2}{ V(x,\omega)} u_3     &=& r_3-\frac{k}{ V(x,\omega)}r_1
		\end{array} \right\} {\rm ~on~} (-\infty,0),\ (0, d) {\rm ~and~} (d,\infty),
		\eeq
		and $u_1:=\frac{\ri}{V(x,\omega)}\left(r_1-\ri k u_3\right)$.

		Analogously to the  proof of Proposition \ref{T:resolv-1D}, $L^2(\R)$-solutions $u$ (disregarding the interface conditions) exist if and only if $\Real(\mu_\pm)>0$, i.e.
		$$k^2-W_\pm(\omega) \notin (-\infty,0],$$
		in other words if and only if $\omega \notin M_+^{(k)}\cup M_-^{(k)}$. The corresponding $\tilde{u}:=(u_2,u_3)^T$ has the form
		\beq\label{3layers_solu_pm}
		\tilde{u}(x)=
		\begin{cases}
			C_-e^{\mu_- x}\begin{pmatrix}
				\mu_- \\ -\ri V_-(\omega)
			\end{pmatrix} + \tilde{u}_p^-(x), & x<0, \\
			C_*^{(1)}e^{\mu_*x} \begin{pmatrix}
				\mu_*\\ -\ri V_*(\omega)
			\end{pmatrix}
			+C_*^{(2)}e^{-\mu_*x}
			\begin{pmatrix}
				\mu_*\\ \ri V_*(\omega)
			\end{pmatrix}
			+ \tilde{u}_p^*(x), & x\in (0,d),\\
			C_+e^{-\mu_+ x}\begin{pmatrix}
				\mu_+ \\ \ri V_+(\omega)
			\end{pmatrix} + \tilde{u}_p^+(x), \quad & x>d,
		\end{cases}
		\eeq

		where  $C_\pm, C_*^{(1,2)}\in \C$ are arbitrary,
		\beq\label{E:up-3layers}
		\begin{aligned}
			\tilde{u}_p^-(x)&=\frac{1}{2}\begin{pmatrix}
				e^{\mu_-x}\mu_-\int_{x}^{0}\rho_-^{(1)}(s)e^{-\mu_-s}\dd s+e^{-\mu_-x}\mu_-\int_{-\infty}^{x}\rho_-^{(2)}(s)e^{\mu_-s}\dd s\\
				-e^{\mu_-x}\ri V_-(\omega)	\int_{x}^{0}\rho_-^{(1)}(s)e^{-\mu_-s}\dd s+e^{-\mu_-x}\ri V_-(\omega)\int_{-\infty}^{x}\rho_-^{(2)}(s)e^{\mu_-s}\dd s\end{pmatrix},\\
			\tilde{u}_p^*(x)&= \frac{1}{2}\begin{pmatrix}e^{\mu_*x}\mu_*\int_{x}^{d}\rho_*^{(1)}(s)e^{-\mu_*s}\dd s+e^{-\mu_*x}\mu_*\int_{0}^{x}\rho_*^{(2)}(s)e^{\mu_*s}\dd s\\
				-e^{\mu_*x}\ri V_*(\omega)	\int_{x}^{d}\rho_*^{(1)}(s)e^{-\mu_*s}\dd s+e^{-\mu_*x}\ri V_*(\omega)\int_{0}^{x}\rho_*^{(2)}(s)e^{\mu_*s}\dd s
			\end{pmatrix},\\
			\tilde{u}_p^+(x)&= \frac{1}{2}\begin{pmatrix}
				e^{\mu_+x}\mu_+\int_{x}^{\infty}\rho_+^{(1)}(s)e^{-\mu_+s}\dd s+e^{-\mu_+x}\mu_+\int_{d}^{x}\rho_+^{(2)}(s)e^{\mu_+s}\dd s\\
				-e^{\mu_+x}\ri V_+(\omega)	\int_{x}^{\infty}\rho_+^{(1)}(s)e^{-\mu_+s}\dd s+e^{-\mu_+x}\ri V_+(\omega)\int_{d}^{x}\rho_+^{(2)}(s)e^{\mu_+s}\dd s
			\end{pmatrix},
		\end{aligned}
		\eeq
		and
		\beq\label{E:rho-3layers}
		\rho_{\pm, *}^{(1)}(x):=\ri\frac{r_2(x)}{V_{\pm, *}(\omega)}+\frac{kr_1(x)}{V_{\pm, *}(\omega)\mu_{\pm, *}}-\frac{r_3(x)}{\mu_{\pm, *}},\quad \rho_{\pm, *}^{(2)}(x):=\ri\frac{r_2(x)}{V_{\pm, *}(\omega)}-\frac{kr_1(x)}{V_{\pm, *}(\omega)\mu_{\pm, *}}+\frac{r_3(x)}{\mu_{\pm, *}}.
		\eeq

		The regularity $u_2,u_3\in H^1(I_j)$, $j=1,2,3$, follows just like at the end of the proof of Proposition \ref{T:resolv-1D} and $u_1\in L^2(\R)$ holds thanks to $V_{\pm,*}(\omega)\neq 0$ and $r_1,u_3\in L^2(\R)$.

		To show that $u\in D(A)$, it remains to enforce the jump conditions $\llbracket u_2 \rrbracket =\llbracket u_3\rrbracket=0$ at $x=0,d$. Since
		\beq\label{3layers_IFC_u_2}
		\begin{aligned}
			\displaystyle\lim_{x\to 0^+}u_2(x)&=\left(C_*^{(1)}+C_*^{(2)}\right)\mu_*+\frac{\mu_*}{2}\int_{0}^{d}\rho_*^{(1)}(s)e^{-\mu_*s}\dd s,\\
			\lim_{x\to 0^-}u_2(x)&=	C_-\mu_-+\frac{\mu_-}{2}\int_{-\infty}^{0}\rho_-^{(2)}(s)e^{\mu_-s}\dd s,
		\end{aligned}
		\eeq
		and $V_*(\omega)\neq 0$, the condition $\llbracket u_2 \rrbracket =0$ at $x=0$ is equivalent to
		\begin{equation}\label{3layers_C_star_1}
			\left(C_*^{(1)}+C_*^{(2)}\right)\mu_*V_*(\omega)=\mu_-V_*(\omega)C_-+\frac{\mu_-V_*(\omega)}{2}\int_{-\infty}^{0}\rho_-^{(2)}(s)e^{\mu_-s}\dd s-\frac{\mu_*V_*(\omega)}{2}\int_{0}^{d}\rho_*^{(1)}(s)e^{-\mu_*s}\dd s.
		\end{equation}
		Similarly, as $\mu_*\neq 0$, the interface condition $\llbracket u_3\rrbracket=0$ (at $x=0$), is equivalent to
		\beq\label{3layers_C_star_2}
		\left(C_*^{(2)}-C_*^{(1)}\right)\mu_*V_*(\omega)=	-\mu_*V_-(\omega)C_- +\frac{\mu_* V_-(\omega)}{2}\int_{-\infty}^{0}\rho_-^{(2)}(s)e^{\mu_-s}\dd s+\frac{\mu_*V_*(\omega)}{2}\int_{0}^{d}\rho_*^{(1)}(s)e^{-\mu_*s}\dd s.
		\eeq
		Combining \eqref{3layers_C_star_1} and \eqref{3layers_C_star_2}, we get
		\beq\label{C_star_1}
		\begin{split}
			\mu_*V_*(\omega)C_*^{(1)}&=\frac{\mu_-V_*(\omega)+\mu_*V_-(\omega)}{2}C_-+\frac{\mu_-V_*(\omega)-\mu_*V_-(\omega)}{4}\int_{-\infty}^{0}\rho_-^{(2)}(s)e^{\mu_-s}\dd s\\
			&-\frac{\mu_*V_*(\omega)}{2}\int_{0}^{d}\rho_*^{(1)}(s)e^{-\mu_*s}\dd s,
		\end{split}
		\eeq
		\beq\label{C_star_2}
		\mu_*V_*(\omega)C_*^{(2)}=\frac{\mu_-V_*(\omega)-\mu_*V_-(\omega)}{2}C_- +\frac{\mu_-V_*(\omega)+\mu_* V_-(\omega)}{4}\int_{-\infty}^{0}\rho_-^{(2)}(s)e^{\mu_-s}\dd s.
		\eeq
		Analogously to the above, the interface conditions $\llbracket u_2\rrbracket=0$ and $\llbracket u_3\rrbracket=0$ at $x=d$ are
		\begin{equation}\label{3layers_C_plus_d}
			\begin{split}
				e^{-\mu_+d}\mu_+V_+(\omega)C_+&=\left(C_*^{(1)}e^{\mu_{*}d}+C_*^{(2)}e^{-\mu_{*}d}\right)\mu_*V_+(\omega)+\frac{\mu_*V_+(\omega)}{2}e^{-\mu_*d}\int_{0}^{d}\rho_*^{(2)}(s)e^{\mu_*s}\dd s\\
				&-\frac{\mu_+V_+(\omega)}{2}e^{\mu_+d}\int_{d}^{\infty}\rho_+^{(1)}(s)e^{-\mu_+s}\dd s.
			\end{split}
		\end{equation}
		and 	\beq\label{3layers_C_star_2_d}
		\begin{split}
			&	e^{-\mu_+d}\mu_+ V_+(\omega)C_+=\left(C_*^{(2)}e^{-\mu_*d}-C_*^{(1)}e^{\mu_*d}\right)\mu_+ V_*(\omega)+\frac{ \mu_+ V_*(\omega)}{2}e^{-\mu_*d}\int_{0}^{d}\rho_*^{(2)}(s)e^{\mu_*s}\dd s\\
			&+\frac{\mu_+  V_+(\omega)}{2}e^{\mu_+d}\int_d^{\infty}\rho_+^{(1)}(s)e^{-\mu_+s}\dd s,
		\end{split}
		\eeq
		respectively.

		Combining \eqref{3layers_C_plus_d} and \eqref{3layers_C_star_2_d}, we get
		\beq\label{C_plus_d}\begin{split}
			&e^{\mu_*d}\left(\mu_*V_+(\omega)+\mu_+V_*(\omega)\right)\mu_*V_*(\omega)C_*^{(1)}+e^{-\mu_*d}\left(\mu_*V_+(\omega)-\mu_+V_*(\omega)\right)\mu_*V_*(\omega)C_*^{(2)}\\
			&=\frac{ \mu_+V_*(\omega)-\mu_*V_+(\omega)}{2}\mu_*V_*(\omega)e^{-\mu_*d}\int_{0}^{d}\rho_*^{(2)}(s)e^{\mu_*s}\dd s+{\mu_+V_+(\omega)}\mu_*V_*(\omega)e^{\mu_+d}\int_{d}^{\infty}\rho_+^{(1)}(s)e^{-\mu_+s}\dd s,
		\end{split}
		\eeq
		Plugging \eqref{C_star_1} and \eqref{C_star_2} into \eqref{C_plus_d}, we obtain
		\begin{align}\label{3layers_dis_eq}
			&\left[e^{\mu_*d}\left(\mu_*V_+(\omega)+\mu_+V_*(\omega)\right)\left(\mu_-V_*(\omega)+\mu_*V_-(\omega)\right)+e^{-\mu_*d}\left(\mu_*V_+(\omega)-\mu_+V_*(\omega)\right)\left(\mu_-V_*(\omega)-\mu_*V_-(\omega)\right)\right]C_-\nonumber\\
			&=(\mu_+V_*(\omega)-\mu_*V_+(\omega))\mu_*V_*(\omega)e^{-\mu_*d}\int_{0}^{d}\rho_*^{(2)}(s)e^{\mu_*s}\dd s+2{\mu_+V_+(\omega)}\mu_*V_*(\omega)e^{\mu_+d}\int_{d}^{\infty}\rho_+^{(1)}(s)e^{-\mu_+s}\dd s\nonumber\\
			&+(\mu_+V_*(\omega)+\mu_*V_+(\omega))\mu_*V_*(\omega)e^{\mu_*d}\int_{0}^{d}\rho_*^{(1)}(s)e^{-\mu_*s}\dd s\nonumber\\
			&-\frac{1}{2}\left[e^{\mu_*d}\left(\mu_*V_+(\omega)+\mu_+V_*(\omega)\right)\left(\mu_-V_*(\omega)-\mu_*V_-(\omega)\right)\right. \nonumber\\
			&\qquad \left.+e^{-\mu_*d}\left(\mu_*V_+(\omega)-\mu_+V_*(\omega)\right)\left(\mu_-V_*(\omega)+\mu_*V_-(\omega)\right)\right]\int_{-\infty}^{0}\rho_-^{(2)}(s)e^{\mu_-s}\dd s.
		\end{align}
		Equation \eqref{3layers_dis_eq} has a unique solution $C_-\in\C$ if and only if
		\beq\label{3layers_dis_relation_0}
		e^{\mu_*d}\left(\mu_*V_+(\omega)+\mu_+V_*(\omega)\right)\left(\mu_-V_*(\omega)+\mu_*V_-(\omega)\right)+e^{-\mu_*d}\left(\mu_*V_+(\omega)-\mu_+V_*(\omega)\right)\left(\mu_-V_*(\omega)-\mu_*V_-(\omega)\right)\neq 0,
		\eeq
		which holds because $\omega\notin N^{(k)}\cup O^{(k)}$. Note that \eqref{cond_d_cD|M|cD-2} covers indeed all cases in which the factors of the exponential functions in \eqref{3layers_dis_relation_0} vanish. The cases $\mu_*=-\mu_+\frac{V_*}{V_+}=\mu_+\frac{V_*}{V_+}$ and $\mu_*=-\mu_-\frac{V_*}{V_-}=\mu_-\frac{V_*}{V_-}$ are both impossible because $\mu_\pm, V_*(\omega) \neq 0.$

		To prove that $\|u\|_{L^2(\R^,\C^3)}\leq c \|r\|_{L^2(\R,\C^3)}$ with $c$ independent of $r$, we use again (like in the  proof of Proposition \ref{T:resolv-1D}) estimates \eqref{E:est-rho-int} together with the obvious analogy for the bounded interval $(0,d)$, i.e.
		$$
		\begin{aligned}
			\left\|e^{\mu_* \cdot }\int_{\cdot}^d\rho_*^{(1)}(s)e^{-\mu_*s}\dd s\right\|_{L^2((0,d))}&\leq c\left\|\rho_*^{(1)}\right\|_{L^2((0,d))}, \\
			\left\|e^{-\mu_* \cdot }\int_{0}^{\cdot}\rho_*^{(2)}(s)e^{\mu_*s}\dd s\right\|_{L^2((0,d))}&\leq c\left\|\rho_*^{(2)}\right\|_{L^2((0,d))}.
		\end{aligned}
		$$
	\end{proof}

	Next, we prove an analogy to Proposition \ref{T:pt-spec-1D}, i.e., we determine the point spectrum of $L_k$ outside $\Omega_0$.
	\begin{prop}\label{3layers_T:pt-spec-1D}
		Let $k \in \R$.
		\beq\label{3layers_E:pt-sp1Dknz-omnz}
		\sigma_p(L_k)\setminus \Omega_0 = N^{(k)}\cup O^{(k)}
		\eeq
		with $N^{(k)}$ given in \eqref{E:Nk_3layers} and $O^{(k)}$ given in  \eqref{E:Ok_3layers}.
		All eigenvalues in $\sigma_p(L_k)\setminus \Omega_0$ are geometrically simple. Eigenvalue $\omega\in N^{(k)}$ is algebraically simple if
		$$
		\begin{aligned}
			0\neq \alpha(\omega):=&\frac{\mu_*^2-V_*^2+k^2}{\mu_*^2V_*}\frac{(V_*^2\mu_+\mu_--\mu_*^2V_+V_-)(\mu_-V_++\mu_+V_-)}{\mu_*^2V_-^2-\mu_-^2V_*^2}\\
			&+\frac{\mu_+^2-V_+^2+k^2}{2\mu_+^2V_*}(3\mu_+V_*-\mu_*V_+)+(\mu_*^2V_+^2-\mu_+^2V_*^2)\left(\frac{1}{\mu_-}\frac{\mu_-^2-V_-^2+k^2}{\mu_*^2V_-^2-\mu_-^2V_*^2}-d\frac{\mu_*^2+V_*^2-k^2}{\mu_*^2V_*^2}\right),
		\end{aligned}
		$$
		where we recall that $V_{\pm,*}=V_{\pm,*}(\omega), \mu_{\pm,*}=\mu_{\pm,*}(\omega).$
	\end{prop}
	\begin{proof}
	Let $\omega \in \Omega\setminus \Omega_0$ and $k\in \R$.	Applying the arguments of the proof of Proposition \ref{T:resolv-1D_3layers} to the case $r=0$, we get that $L^2(\R)$-solutions $\psi$ are given by
		\beq\label{E:efn-3layers}
		\begin{aligned}
			\begin{pmatrix} \psi_2\\ \psi_3 \end{pmatrix}(x)&=\begin{cases} \tilde{A}\bspm \mu_+\\ \ri V_+(\omega)\espm e^{-\mu_+ x},  &\quad x>d,\\
				\tilde{B}\bspm \mu_*\\ \ri V_*(\omega)\espm e^{-\mu_* x}+	\tilde{C}\bspm \mu_*\\ -\ri V_*(\omega)\espm e^{\mu_* x}&\quad 0<x<d,\\
				\tilde{D}\bspm \mu_-\\ -\ri V_-(\omega)\espm e^{\mu_- x},  &\quad x<0,
			\end{cases}\\
			\psi_1(x) &= \frac{k\psi_3(x)}{V(\omega)}
		\end{aligned}
		\eeq
		with free constants $\tilde{A}$, $\tilde{B}$, $\tilde{C}$, $\tilde{D}\in \C$. The $L^2$-property holds if and only if $\omega \notin M_+^{(k)} \cup M_-^{(k)}$. Once again, $u_{2,3}\in H^1(I_j), j=1,2,3$,  holds. We can normalize $\psi$ so  that $\tilde{D}=1$. Indeed, if $\tilde{D}=0$, then using the interface conditions, one can show that $\tilde{A}=\tilde{B}=0$ and hence $\psi\equiv 0.$

		Next, we consider the interface conditions \eqref{IFCs}. The conditions $\llbracket \psi_2\rrbracket =0$ and $\llbracket \psi_3\rrbracket =0$ (at $x=0$ and $x=d$) yield
		\beq \label{3layers_E:ABrel}
		\begin{split}
			\mu_-=(\tilde{B}+\tilde{C})\mu_*,	\quad &V_-(\omega)= \left(\tilde{C}-\tilde{B}\right) V_*(\omega), \quad \tilde{A}e^{-\mu_+d}\mu_+=(\tilde{B}e^{-\mu_*d}+\tilde{C}e^{\mu_*d})\mu_*,\\
			&	\tilde{A}e^{-\mu_+d}V_+(\omega)=\left(\tilde{B}e^{-\mu_*d}-\tilde{C}e^{\mu_*d}\right) V_*(\omega),\quad
		\end{split}
		\eeq

		The first equation in \eqref{3layers_E:ABrel} yields the necessary condition $\mu_*\neq 0$, i.e. $W_*(\omega)\neq k^2$. Otherwise $\mu_-=0$ and $\psi_3\notin L^2(\R,\C).$ The first three conditions in \eqref{3layers_E:ABrel} now produce a unique set of coefficients $\tilde{A}, \tilde{B}, \tilde{C}$, namely
		$$
		\begin{aligned}
			\tilde{A}&=\frac{e^{\mu_+ d}}{2}\frac{\mu_*}{\mu_+}\left(e^{\mu_*d}\left(\frac{\mu_-}{\mu_*}+\frac{V_-}{V_*}\right)+e^{-\mu_*d}\left(\frac{\mu_-}{\mu_*}-\frac{V_-}{V_*}\right)\right),\\
			\tilde{B}&=	\frac{1}{2}\left(\frac{\mu_-}{\mu_*}-\frac{V_-}{V_*}\right), \quad \tilde{C}=\frac{1}{2}\left(\frac{\mu_-}{\mu_*}+\frac{V_-}{V_*}\right).
		\end{aligned}
		$$
		The last condition in \eqref{3layers_E:ABrel} for a nontrivial solution (i.e. for $\tilde{B}, \tilde{C}\neq 0$) is equivalent to
		\beq\label{3layers_E:ABrel2}
		e^{\mu_*d}\left(\mu_*V_+(\omega)+\mu_+V_*(\omega)\right)\left(\mu_-V_*(\omega)+\mu_*V_-(\omega)\right)+e^{-\mu_*d}\left(\mu_*V_+(\omega)-\mu_+V_*(\omega)\right)\left(\mu_-V_*(\omega)-\mu_*V_-(\omega)\right)= 0,
		\eeq
		i.e. \eqref{cond_d_cD|M|cD} or \eqref{cond_d_cD|M|cD-2}, i.e. $\omega \in N^{(k)}\cup O^{(k)}$. An eigenfunction $\varphi_0$ with $\|\varphi_0\|=1$ is, of course, $\varphi_0:=\psi/\|\psi\|$.

		The construction of the above solution $\psi\in D(A)$ guarantees the geometric simplicity of any eigenvalue  in the sense that $\lambda=0$ is a geometrically simple eigenvalue of  $L_k u=\lambda u$.

		Next we show that if $\alpha(\omega)\neq 0$, then the eigenvalue $\omega\in N^{(k)}$ is also algebraically simple, which by \eqref{E:gen-evec-eq} means that the problem
		\begin{equation}\label{Lpb_Jordan}	L_k(\omega)u=\psi,
		\end{equation}
		(with  $\psi$ as in \eqref{E:efn-3layers}) has no solution $u\in D(A)$. Let us set  $r:=\psi\in L^2(\R,\C^3)$. Assuming for a contradiction that $u\in D(A)$ is a solution of \eqref{Lpb_Jordan}, the variation of constants from the proof of Proposition \ref{T:resolv-1D_3layers} gives the explicit form \eqref{3layers_solu_pm}, \eqref{E:up-3layers}, and \eqref{E:rho-3layers}  with $u_1=\frac{\ri}{V(x,\omega)}\left(r_1-\ri k u_3\right)$.

		In order for $u$  to belong to $D(A)$, $u_{2,3}$ must satisfy the continuity at $x=0$ and $x=d$. This implies that the constants $C_-$, $C_*^{(1)}$, $C_*^{(2)}$, $C_+$ have to solve the linear system
		$$T\left(C_-, C_*^{(1)}, C_*^{(2)}, C_+\right)^\trans=b,$$
		where
		$$T:=\left(\begin{matrix}
			-\mu_- & \mu_* & \mu_* & 0 \\
			V_-(\omega) & -V_*(\omega) & V_*(\omega) & 0 \\
			0 & -\mu_*e^{\mu_* d} & -\mu_*e^{-\mu_* d} & \mu_+e^{-\mu_+d} \\
			0 & V_*(\omega) e^{\mu_* d} & -V_*(\omega) e^{-\mu_* d} & V_+(\omega)e^{-\mu_+d}
		\end{matrix}\right)$$
		and
		$$b:=\frac{1}{2}\left(\begin{matrix}
			\mu_-\int_{-\infty}^{0}\rho_-^{(2)}(s)e^{\mu_-s}\dd s-\mu_*\int_{0}^{d}\rho_*^{(1)}(s)e^{-\mu_*s}\dd s\\
			V_-(\omega)\int_{-\infty}^{0}\rho_-^{(2)}(s)e^{\mu_-s}\dd s+ V_*(\omega)\int_{0}^{d}\rho_*^{(1)}(s)e^{-\mu_*s}\dd s\\
			e^{-\mu_*d}\mu_*\int_{0}^{d}\rho_*^{(2)}(s)e^{\mu_*s}\dd s-e^{\mu_+d}\mu_+\int_{d}^{\infty}\rho_+^{(1)}(s)e^{-\mu_+s}\dd s\\
			e^{-\mu_*d} V_*(\omega)\int_{0}^{d}\rho_*^{(2)}(s)e^{\mu_*s}\dd s+e^{\mu_+d} V_+(\omega)	\int_{d}^{\infty}\rho_+^{(1)}(s)e^{-\mu_+s}\dd s
		\end{matrix}\right).$$
		Note that $T$ is singular since $T(1,\tilde{C},\tilde{B},\tilde{A})^\trans=0$, as dictated by \eqref{3layers_E:ABrel}.
		The functions $\rho^{(1,2)}_{*,\pm}$ in $b$ are as in \eqref{E:rho-3layers} with $r=\psi.$

		In order to find a contradiction and exclude the existence of a solution $u\in D(A)$ of \eqref{Lpb_Jordan}, we now prove that $b$ is \textit{not} orthogonal to the kernel of $\overline T^\trans$ if $\alpha(\omega)\neq 0$. Standard computations show that $\ker\overline T^\trans$ is one-dimensional and given by
		\begin{equation*}
			\ker\overline T^\trans=\spann\,p, \quad p:=\begin{pmatrix} e^{\overline{\mu_*}d}\overline{V_-}(\overline{\mu_*}\overline{V_+}+\overline{V_*}\overline{\mu_+})\\
				e^{\overline{\mu_*}d}\overline{\mu_-}(\overline{\mu_*}\overline{V_+}+\overline{V_*}\overline{\mu_+})\\
				\overline{V_+}(\overline{\mu_*}\overline{V_-}-\overline{\mu_-}\overline{V_*})\\
				-\overline{\mu_+}(\overline{\mu_*}\overline{V_-}-\overline{\mu_-}\overline{V_*})\\		\end{pmatrix}.
		\end{equation*}
		For the scalar product $\left(b,p\right)_{\C^4}$ a direct calculation shows that
		\beq\label{E:bp-IP}
		\begin{aligned}
		-4\ri\left(b,p\right)_{\C^4}=& e^{\mu_* d} (\mu_*V_++\mu_+V_*) \left(\frac{2}{\mu_-}(\mu_-^2-V_-^2+k^2) + \frac{d}{\mu_*^2V_*^2}(\mu_-^2V_*^2-\mu_*^2V_-^2)(\mu_*^2+V_*^2-k^2) \right) \\
		&+ \frac{de^{-\mu_* d}}{\mu_*^2V_*^2}(\mu_-V_*-\mu_*V_-)^2(\mu_+V_*-\mu_*V_+)(\mu_*^2+V_*^2-k^2) \\
		&+\frac{\sinh(\mu_* d)}{\mu_*^3V_*^2}\left((\mu_-V_*-\mu_*V_-)^2(\mu_*V_++\mu_+V_*) + (\mu_*^2V_-^2-\mu_-^2V_*^2)(\mu_*V_+-\mu_+V_*))\right)(\mu_*^2-V_*^2+k^2)\\
		&+\frac{2}{\mu_+^2V_*}(\mu_-V_*\cosh(\mu_*d)+\mu_*V_-\sinh(\mu_* d))(\mu_-V_*-\mu_*V_-)(\mu_+^2-V_+^2+k^2).
		\end{aligned}
		\eeq
		A number of simplifications, using \eqref{E:cond-3-layers-B} repeatedly, shows that after multiplication with $\frac{e^{\mu_*d}}{2}\frac{\mu_*V_++\mu_+V_*}{(\mu_*V_--\mu_-V_*)^2}$ (which is non-zero since $\omega \notin O^{(k)}$) the right hand side equals $\alpha(\omega)$.
	\end{proof}
	\begin{rem}
		One can easily check in \eqref{E:bp-IP} that $(b,p)_{\C^4}=0$ if $\omega \in O^{(k)}$ and the first case in \eqref{cond_d_cD|M|cD-2} is satisfied. We expect $(b,p)_{\C^4}$ to vanish also in the second case of \eqref{cond_d_cD|M|cD-2}.
		\end{rem}

	\begin{rem}
		Note that it can be possibly proved that $\alpha(\omega)\neq 0$ for all $\omega \in N^{(k)}$. We refrain from trying to show this and only check that $\alpha(\omega) \neq 0$ numerically for specific material parameters, see Example \ref{Ex:3layers-eigs}.
	\end{rem}

	Finally, we study isolatedness of eigenvalues in $\sigma_p(L_k)\setminus \Omega_0$. Again, we restrict our attention to eigenvalues in $N^{(k)}$.
	\begin{prop}\label{3layers_T:disc_sp_1D}
		Let $k \in \R$. Eigenvalue $\omega \in N^{(k)}$ is isolated if and only if
		$$d\neq \beta(\omega),$$
		where
		$$
		\begin{aligned}
			\beta(\omega):=&\frac{1}{(\omega+V_*)(V_*^2\mu_-^2-V_-^2\mu_*^2)(V_*^2\mu_+^2-V_+^2\mu_*^2)}\left\{(V_*(\omega+V_*)-2\mu_*^2)(\mu_*^2V_+V_--V_*^2\mu_+\mu_-)(\mu_+V_-+\mu_-V_+)\right.\\
			& \left. -\frac{\mu_*^2V_*}{\mu_+\mu_-}\left[\mu_-V_+(\omega+V_+)(\mu_*^2V_-^2-V_*^2\mu_-^2)+\mu_+V_-(\omega+V_-)(\mu_*^2V_+^2-V_*^2\mu_+^2)\right]\right.\\
			&\left. + 2\mu_*^2V_*\left[\mu_*^2(\mu_+V_-^2+\mu_-V_+^2)-V_*^2\mu_+\mu_-(\mu_++\mu_-)\right]\right\}.
		\end{aligned}
		$$
		Here, we have used the short notation $V_{*,\pm}:=V_{*,\pm}(\omega)$ again.

		In summary (together with Proposition \ref{3layers_T:pt-spec-1D}) we have for all $k\in \R$
		$$\{\omega\in N^{(k)}: \alpha(\omega)\neq 0, \beta(\omega) \neq d\}\subset \sigma_d(L_k)\setminus \Omega_0.$$
	\end{prop}

	\begin{proof}
		Analogously to Proposition \eqref{T:disc_sp_1D} we choose $\omega \in N^{(k)}$ and need to show the existence of $\delta>0$ such that
		$$\lambda \in B_\delta(0)\setminus\{0\}\subset \C \ \Rightarrow \ \lambda \in \rho(L_k(\omega)).$$
		Hence, we work with the linear pencil $L_k^{(\lambda)}(\omega):=L_k(\omega)-\lambda I = A-(B(\omega)+\lambda I)$ and need to show
	$$\omega  \in \rho(L_k^{(\lambda)}) \quad \forall \lambda \in B_\delta(0)\setminus\{0\}\subset \C .$$
Let us define
$$\mu_{*,\pm}(\lambda):=\sqrt{k^2+(\omega-\ri \lambda)(V_{*,\pm}(\omega)-\ri \lambda)}, \quad  V_{*,\pm}(\omega,\lambda):=V_{*,\pm}(\omega)-\ri\lambda.$$

		Based on Proposition \ref{T:resolv-1D_3layers} we have
		$$\rho(L_k^{(\lambda)})\supset \Omega\setminus (M^{(k,\lambda)}_+\cup M^{(k,\lambda)}_-\cup N^{(k,\lambda)} \cup O^{(k,\lambda)}\cup \Omega_0^{(\lambda)}),$$
		where
	\[
			\begin{split}
			M^{(k,\lambda)}_{\pm}:= &~\{ \omega \in  \Omega \setminus \Omega_0: -(\omega-\ri \lambda)V_{\pm}(\omega,\lambda)\in [k^2,\infty)\}, \\
			N^{(k,\lambda)}:= & ~\{ \omega \in \Omega \setminus (\Omega_0 \cup  M^{(k,\lambda)}_{+}\cup M^{(k,\lambda)}_{-}):  -(\omega-\ri \lambda)V_{*}(\omega,\lambda)\neq k^2, \eqref{cond_d_cD|M|cD-lam} \text{ holds and } \eqref{cond_d_cD|M|cD-2-lam} \text{ does not hold}  \},\\
			O^{(k,\lambda)}:= &~\{ \omega \in  \Omega \setminus (\Omega_0 \cup  M^{(k,\lambda)}_{+}\cup M^{(k,\lambda)}_{-}): \eqref{cond_d_cD|M|cD-2-lam} \text{ holds}  \},\\
			\Omega_0^{(\lambda)}:=& ~\{\omega\in \Omega:\omega-\ri \lambda=0 \text{ or } V_-(\omega,\lambda)=0 \ \text{or} \  V_*(\omega,\lambda)=0 \ \text{or} \ V_+(\omega,\lambda)=0\},
		\end{split}
		\]
		\beq\label{cond_d_cD|M|cD-lam}
			e^{2\mu_*(\lambda)d}=\frac{\left(\mu_*(\lambda)V_+(\omega,\lambda)-\mu_+(\lambda)V_*(\omega,\lambda)\right)\left(\mu_*(\lambda)V_-(\omega,\lambda)-\mu_-(\lambda)V_*(\omega,\lambda)\right)}{\left(\mu_*(\lambda)V_+(\omega,\lambda)+\mu_+(\lambda)V_*(\omega,\lambda)\right)\left(\mu_-(\lambda)V_*(\omega,\lambda)+\mu_*(\lambda)V_-(\omega,\lambda)\right)}=:g(\lambda),
		\eeq
			\beq\label{cond_d_cD|M|cD-2-lam}
		\begin{cases}
			&\mu_*(\lambda)V_+(\omega,\lambda)+\mu_+(\lambda)+V_*(\omega,\lambda)=0,\   \mu_*(\lambda)V_-(\omega,\lambda)-\mu_-(\lambda)V_*(\omega,\lambda)=0\\
			&\text{or} \\
			& \mu_*(\lambda)V_+(\omega,\lambda)-\mu_+(\lambda)V_*(\omega,\lambda)=0,\  \mu_*(\lambda)V_-(\omega,\lambda)+\mu_-(\lambda)V_*(\omega,\lambda)=0.
		\end{cases}
		\eeq
		Clearly, because $\omega \notin (M_+^{(k)}\cup M_-^{(k)} \cup O^{(k)}\cup \Omega_0)$, we get $\omega \notin (M_+^{(k,\lambda)}\cup M_-^{(k,\lambda)} \cup O^{(k,\lambda)}\cup \Omega_0)$ for all $\lambda$ small enough.

		It remains to be shown that \eqref{cond_d_cD|M|cD-lam} cannot hold if $\lambda$ is nonzero and small enough. Otherwise there would be a sequence $(\lambda_j)\subset \C$ such that $\lambda_j \to 0$ and $f(\lambda_j):=e^{2\mu_*(\lambda_j)d}-g(\lambda_j)=0 \ \forall j.$ Due to the $C^1$ nature of $f$, this would imply $f'(0)=0.$ Because $f'(0)=2g(0)\mu_*'(0)d-g'(0)=-\ri g(0)d(\omega+V_*(\omega))\frac{1}{\mu_*(0)}-g'(0)$, the equality $f'(0)=0$ is equivalent to
		$$d= \frac{\ri\mu_*(0)g'(0)}{g(0)(\omega+V_*(\omega))}.$$
		A direct calculation shows that $ \frac{\ri\mu_*(0)g'(0)}{g(0)(\omega+V_*(\omega))}=\beta(\omega)$.
	\end{proof}
\bex\label{Ex:3layers-eigs}
	We choose the wave number $k=2$ and the $\PT$-symmetric setting (see Theorem \ref{T:PT_sym} for the definition of $\PT$-symmetry) with three homogenous layers
	 $$V_*(\omega)=-\frac{6}{5}\omega, \ V_-(\omega)=-\omega\left(1-\frac{2\pi}{\omega^2-\ri\frac{\omega}{2}}\right), \ V_+(\omega)=\overline{V_-(\omega)}.$$
	A width $d$ of the middle layer and corresponding eigenvalues can be found using \eqref{cond_d_cD|M|cD}. In Figure \ref{F:d-3layers} we plot $d_j(\omega)$ for $j=-1,0,1$ and $\omega \in (0,5)$. Apparently $d_{-1}(\omega)\notin (0,\infty)$ for all $\omega \in (0,5)$, $d_0(\omega)\in (0,\infty)$ for $\omega \in (0,a)\cup (b,5)$, where $a\approx 2.5$, $b\approx 3.75$, and, finally, $d_1(\omega)\in (0,\infty)$ for $\omega \in (c,5)$ with $c\approx 1.83$.
	\begin{figure}[ht!]
		\centering
			\includegraphics[width=1.0\linewidth]{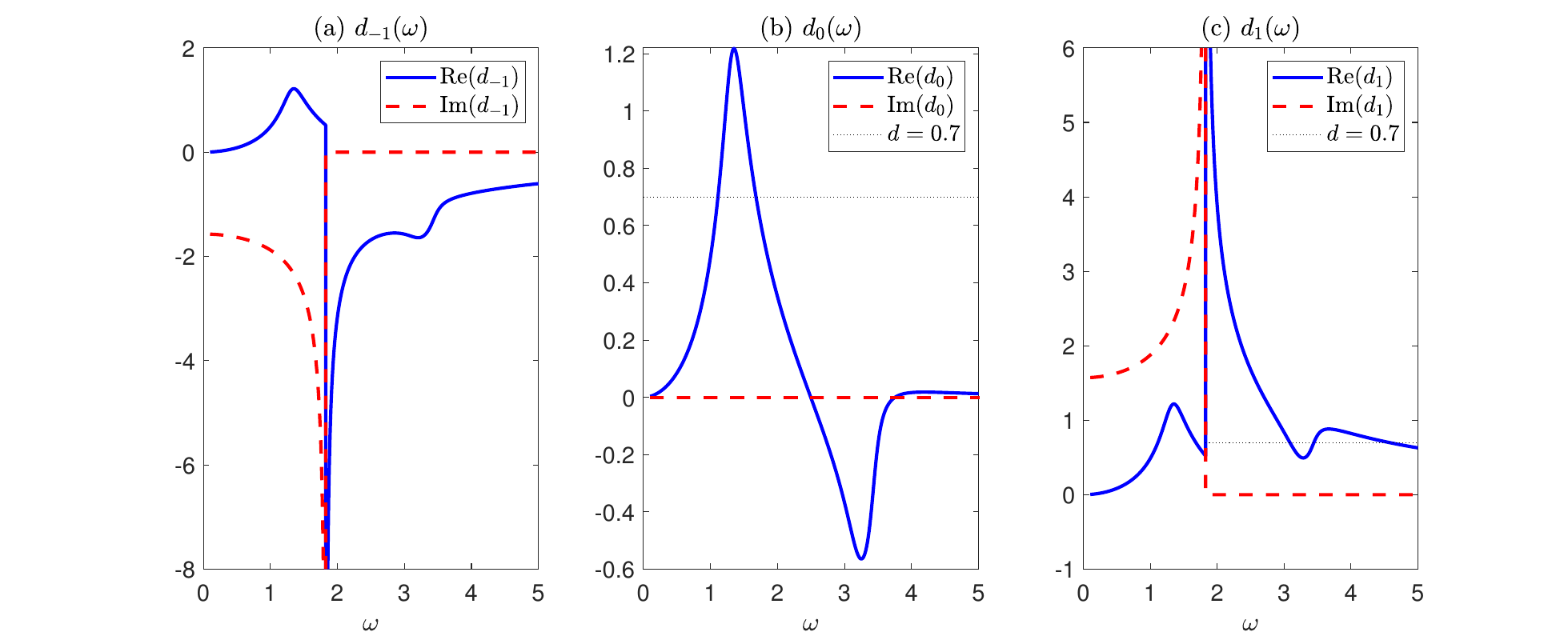}
		\caption{{\footnotesize Functions $d_j(\omega)$ with $j=-1,0,1$ from \eqref{cond_d_cD|M|cD}. }}
		\label{F:d-3layers}
	\end{figure}

	As one can see, for example, for $d=0.7$ there are at least five eigenvalues in $N^{(k)}$, see the intersections with the dotted line $d=0.7$ in Figure  \ref{F:d-3layers} (b) and (c). The approximate values are $\omega_0\in \{1.11, 1.68, 3.09, 3.44 , 4.57\}$, the first two of which are visible in Figure \ref{F:d-3layers} (b) and the last three in Figure \ref{F:d-3layers} (c).

	Next, we set the width of the middle layer equal to $d_1(\omega)$ and check the algebraic simplicity and isolatedness of the corresponding eigenvalue $\omega \in (c,5)$ (i.e. the $x-$coordinate of the point on the graph of the real part in Figure  \ref{F:d-3layers}  (c)). Using Propositions \ref{3layers_T:pt-spec-1D} and \ref{3layers_T:disc_sp_1D}, we plot in Figure \ref{F:alpha-beta-3layers} the quantities $\alpha(\omega)$ and $|d_1(\omega)-\beta(\omega)|$ from Propositions \ref{3layers_T:pt-spec-1D} and \ref{3layers_T:disc_sp_1D}. Clearly, $\alpha$ remains positive on the whole $(c,5)$ and $d_1\neq \beta$ everywhere except for at most two points, namely $\omega \approx 2.15$ and $\omega \approx 3.3$. Hence, the last three eigenvalues at $d=0.7$ mentioned above, i.e. $\{3.09, 3.44 , 4.57\}$ are all algebraically simple and isolated.
	\begin{figure}[ht!]
	\centering
	\includegraphics[width=0.75\linewidth]{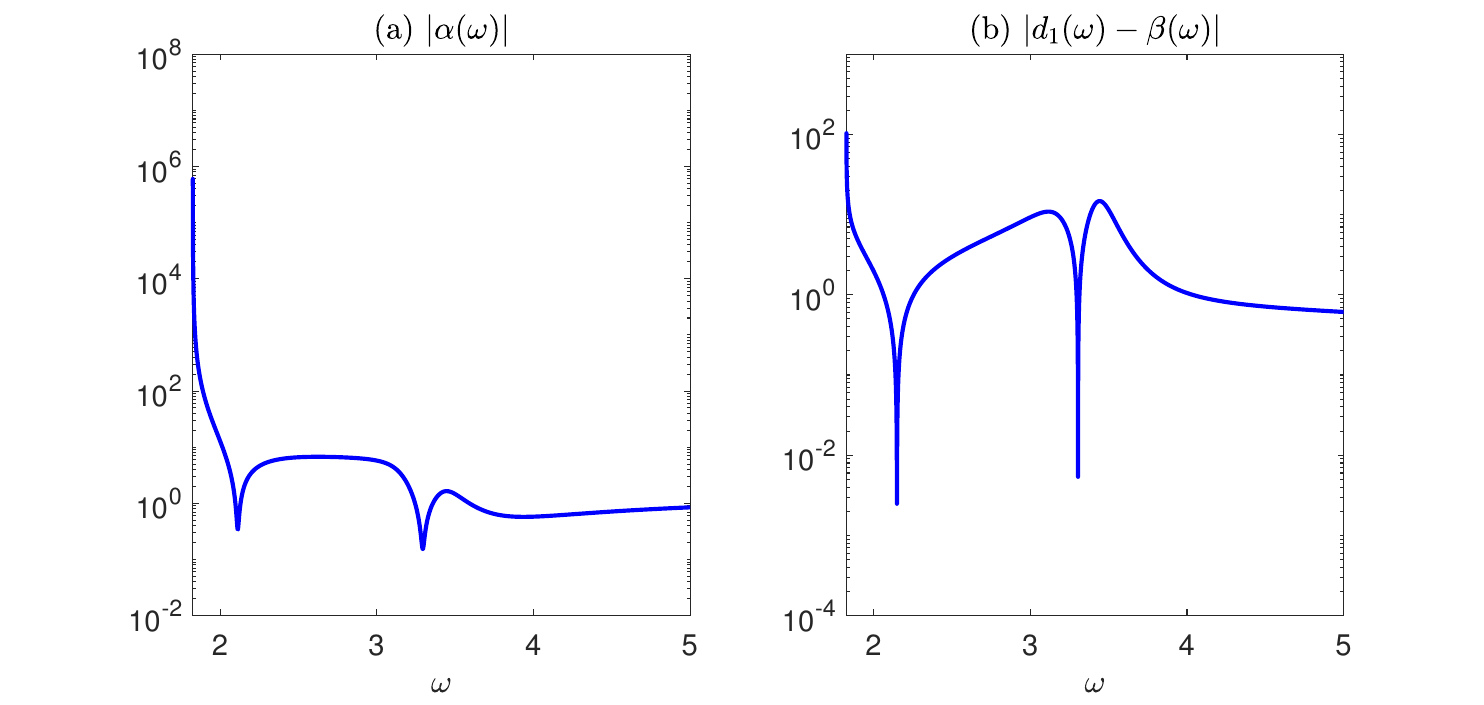}
	\caption{{\footnotesize Quantities  $\alpha(\omega)$ and $|d_1(\omega)-\beta(\omega)|$ with $\alpha$ and $\beta$ from Propositions \ref{3layers_T:pt-spec-1D} and \ref{3layers_T:disc_sp_1D}  respectively.}}
	\label{F:alpha-beta-3layers}
\end{figure}
	\eex


	\section{Bifurcation of nonlinear surface plasmons}\label{bifurc}

	The aim of this section is to prove Theorem \ref{mainthm_loc_bifurc}. As stated in this theorem, the solution of \eqref{system} is expanded in
	\beq\label{omega_varphi}
	\omega = \omega_0+\eps \nu + \eps^2 \sigma, \quad u = \eps^{1/2}\varphi_0 + \eps^{3/2} \phi + \eps^{5/2} \psi,
	\eeq
	where $\omega_0 \in  \C, \varphi_0 \in D(A)$ are fixed by the linear eigenvalue problem (and the normalization $\|\varphi_0\|=1$) and $\nu \in \C, \phi\in D(A)\cap \langle \varphi_0^*\rangle^\perp$ are fixed by \eqref{E:nuA} and \eqref{E:phi-eq-A}.
	Note that $\nu$ is well defined because of the assumption (A-T) and that the normalization $\langle \varphi_0, \varphi_0^*\rangle =1$ is allowed because $\langle \varphi_0, \varphi_0^*\rangle =0$ would imply the solvability of $L_k(\omega_0)v=\varphi_0$ for $v\in D(A)$, which contradicts the algebraic simplicity of $\omega_0$.

	We decompose the nonlinear equation \eqref{system} using the projection
	$$P_0:D(A)\to\langle\varphi_0\rangle, u\mapsto \langle u, \varphi_0^*\rangle \varphi_0$$
	and the complementary projection $Q_0:=I-P_0: D(A)\to \langle \varphi_0^* \rangle^\perp$. This discussion as well as the subsequent fixed point arguments are very similar to those in \cite{dohnal2021eigenvalue}. We include them here (in a more compact form) for completeness.

	Applying $P_0$ to equation \eqref{system}, we get
	\begin{equation}\label{proj_P0_1}
		\begin{split}
			\langle L_k(\cdot,\omega)u,\varphi_0^*\rangle&=\langle h(\cdot,\omega,u),\varphi_0^*\rangle\\
			&=\langle h(\cdot,\omega,u)-\varepsilon^{\frac32}h(\cdot,\omega_0,\varphi_0),\varphi_0^*\rangle+\varepsilon^{\frac32}\langle h(\cdot,\omega_0,\varphi_0),\varphi_0^*\rangle.
		\end{split}
	\end{equation}

	On the other hand, Taylor expanding $B(\omega)$ in $\omega_0$ up to order two (using assumption (A-V)), we have
	\begin{equation*}
		\begin{split}
			\langle L_k(\cdot,\omega)u,\varphi_0^*\rangle&=\langle L_k(\cdot,\omega_0)u,\varphi_0^*\rangle+\langle\left(L_k(\cdot,\omega)-L_k(\cdot,\omega_0)\right)u,\varphi_0^*\rangle\\
			&=\langle u,L_k(\cdot,\omega_0)^*\varphi_0^*\rangle-\langle\left(B(\cdot,\omega)-B(\cdot,\omega_0)\right)u,\varphi_0^*\rangle\\
			&=-\left\langle\left(\partial_\omega B(\cdot,\omega_0)(\omega-\omega_0)+\frac12\partial_\omega^2B(\cdot,\omega_0)(\omega-\omega_0)^2+I(\cdot, \omega)(\omega-\omega_0)^3\right)u,\varphi_0^*\right\rangle,
		\end{split}
	\end{equation*}
	where
	\begin{equation*}
		I(x,\omega):=\frac{1}{2\pi \mathrm{i}}\int_{\partial B_r(\omega_0)}\frac{B(x,z)}{(z-\omega_0)^3(z-\omega)}\dd z
	\end{equation*}
	for any $r<\delta$ and $\omega \in B_r(\omega_0)$.
	Inserting now the expansions of $\omega$ and $u$ from \eqref{omega_varphi} produces
	\begin{equation}\label{proj_P0_2}
		\begin{split}
			-\langle L_k(\cdot,\omega)u,\varphi_0^*\rangle&=\varepsilon^{\frac{3}{2}}\nu\langle\partial_\omega B(\cdot,\omega_0)\varphi_0,\varphi_0^*\rangle+\varepsilon^{\frac{5}{2}}\sigma\langle\partial_\omega B(\cdot,\omega_0)\varphi_0),\varphi_0^*\rangle\\
			&\quad+\varepsilon^{\frac{5}{2}}\bigg\langle\bigg(\nu\partial_\omega B(\cdot,\omega_0)\phi+\frac{\nu^2}2\partial_\omega^2B(\cdot,\omega_0)\varphi_0\bigg),\varphi_0^*\bigg\rangle+\langle v(\cdot,\nu,\sigma,\phi,\psi),\varphi_0^*\rangle,
		\end{split}
	\end{equation}
	where
	\begin{equation}\label{v}
		\begin{split}
			v(\cdot,\nu,\sigma,\phi,\psi):&=\varepsilon^{\frac{7}{2}}\left(\sigma\partial_\omega B(\cdot,\omega_0)\phi+(\nu+\varepsilon\sigma)\partial_\omega B(\cdot,\omega_0)\psi+\frac{\nu^2}2\partial_\omega^2B(\cdot,\omega_0)(\phi+\varepsilon\psi)\right.\\
			&\left.\quad+\frac12\partial_\omega^2B(\cdot,\omega_0)(2\nu\sigma+\varepsilon\sigma^2)(\varphi_0+\varepsilon\phi+\varepsilon^{2}\psi)\right)+\varepsilon^\frac72 I(\cdot,\omega)(\nu+\eps \sigma)^3(\varphi_0+\varepsilon\phi+\varepsilon^{2}\psi)
		\end{split}
	\end{equation}
	with $\omega =\omega_0+\eps \nu+\eps^2\sigma.$

	The first inner product on the right hand side of \eqref{proj_P0_1} is $o(\eps^\frac{3}{2})$ due to the Lipschitz continuity of $h$, see \eqref{E:h-Lip-est}. Comparing now \eqref{proj_P0_1} and \eqref{proj_P0_2}, the terms of order $\eps^\frac{3}{2}$ match if and only if we set
	\begin{equation}\label{nu_GEN}
		\nu:=-\frac{\langle h(\cdot,\omega_0, \varphi_0),\varphi_0^*\rangle}{\langle\partial_\omega B(\cdot,\omega_0)\varphi_0,\varphi_0^*\rangle},
	\end{equation}
	which is well-defined thanks to assumption (A-T). From the rest of \eqref{proj_P0_1}-\eqref{proj_P0_2} we obtain
	\begin{equation}\label{3_sys}
		\begin{split}
			\varepsilon^{\frac{5}{2}}\sigma\langle\partial_\omega B(\cdot,\omega_0)\varphi_0,\varphi_0^*\rangle=&-\varepsilon^{\frac{5}{2}}\Big\langle\left(\nu\partial_\omega  B(\cdot,\omega_0)\phi+\frac{\nu^2}2\partial_\omega^2B(\cdot,\omega_0)\varphi_0\right),\varphi_0^*\Big\rangle\\
			&-\langle v(\cdot,\nu,\sigma,\phi,\psi),\varphi_0^*\rangle-\langle h(\cdot,\omega,u)-\varepsilon^{\frac{3}{2}}h\left(\cdot, \omega_0, \varphi_0\right),\varphi_0^*\rangle.
		\end{split}
	\end{equation}

	Next, we apply $Q_0$ to  \eqref{system}.  First, we write
	\begin{equation*}
		Q_0L_k(\cdot,\omega)u
		=Q_0L_k(\cdot,\omega_0)Q_0(\varepsilon^{\frac{3}{2}}\phi+\varepsilon^{\frac{5}{2}}\psi) -Q_0\big( B(\cdot,\omega)-B(\cdot,\omega_0)\big)u.
	\end{equation*}
	The application of $Q_0$ to \eqref{system} produces
	\begin{equation}\label{E:Q-eq}
		\begin{split}
			&Q_0L_k(\cdot,\omega_0)Q_0(\phi+\varepsilon\psi)=\varepsilon^{-\frac{3}{2}} \left( Q_0h(\cdot,\omega,u)+Q_0\big(B(\cdot,\omega)-B(\cdot,\omega_0)\big)u\right)\\
			&=\varepsilon^{-\frac{3}{2}} Q_0h(\cdot,\omega,u)+ Q_0\left((\nu+\eps\sigma)\pa_\omega B(\cdot,\omega_0) + \eps\frac{(\nu+\eps\sigma)^2}{2}\pa_\omega^2B(\cdot,\omega_0)+\eps^{2}I(\cdot,\omega)(\nu+\eps\sigma)^3\right)(\varphi_0+\eps\phi+\eps^2\psi).
		\end{split}
	\end{equation}

	It remains to solve the system \eqref{3_sys}, \eqref{E:Q-eq} for $\sigma$ and $\psi$ with $\phi\in D(A)\cap \langle \phi_0^*\rangle^\perp$ given by \eqref{E:phi-eq-A}.

	\subsection{Proof of Theorem  \ref{mainthm_loc_bifurc}}

	We take advantage of the fact that the nonlinearity $h$ is well defined in $\cH^1$ due to the algebra property of $H^1(a,b)$ for any $(a,b)\subset \R$. In other words, $h(\cdot,\omega,u)\in \cH^1$ for any $u\in \cH^1$, see Lemma \ref{L:h-lip}.

	First of all, we reduce equation \eqref{E:Q-eq} by eliminating the $O(1)$ part.
	Because $h(x,\omega,\alpha u)=\alpha^3 h(x,\omega,u)$ for all $\alpha\geq 0$ and $(x,\omega,u)\in \R\times B_\delta(\omega_0) \times (D(A)\cap \cH^1)$, the term $h(x,\omega,u)$ in \eqref{E:Q-eq} equals $\eps^{3/2}h(x,\omega,\varphi_0+\varepsilon\phi+\varepsilon^{2}\psi)$. The $O(1)$ part of \eqref{E:Q-eq} thus holds if and only if
	$$	Q_0L_k(\cdot,\omega_0)Q_0\phi=Q_0(h(\cdot, \omega_0, \varphi_0)+\nu\partial_\omega B(\cdot,\omega_0)\varphi_0).$$
	Due to \eqref{nu_GEN} we have  $P_0(h(\cdot, \omega_0, \varphi_0)+\nu\partial_\omega B(\cdot,\omega_0)\varphi_0)=0$ and hence
	\begin{equation}\label{phi_GEN}
		Q_0L_k(\cdot,\omega_0)Q_0\phi=h(\cdot, \omega_0, \varphi_0)+\nu\partial_\omega B(\cdot,\omega_0)\varphi_0,
	\end{equation}
	which is equivalent to \eqref{E:phi-eq-A} in $D(A)\cap \cH^1\cap \langle \phi_0^*\rangle^\perp$ and uniquely solvable by the following lemma. Note that the right hand side is in $\cH^1$ because $\pa_\omega V(\cdot,\omega_0)\in W^{1,\infty}(I_j)$ for all $j$ and $\varphi_0\in \cH^1$ (and hence by the algebra property, also $h(\cdot,\omega,\varphi_0)\in \cH^1$). The regularity of $\varphi_0$ follows from Lemma \ref{L:phi0-reg}.
	\blem\label{L:Lk-inv-smooth}
	$(Q_0L_kQ_0)^{-1}:Q_0\cH^1 \to Q_0(\cH^1\cap D(A))$ is a bounded operator, i.e. $\|(Q_0L_kQ_0)^{-1}\|\leq M_L$ with $M_L>0$.
	\elem
	\bpf
	The boundedness of the inverse $(Q_0L_kQ_0)^{-1}:Q_0L^2 \to Q_0D(A)$ (with $D(A)$ equipped with the graph norm) follows from the closed range theorem \cite[Sec. VI.6]{DS1988}. Indeed, the operator $Q_0L_kQ_0$ is Fredholm (see \cite[Theorem IV.5.28]{kato2013perturbation}, where the fact that $\lambda=0$ is a simple isolated eigenvalue of the eigenvalue problem \eqref{Lpb} is used).

	The $H^1$ regularity on each layer holds for the second and the third component of each element in $D(A)$ by definition of $D(A)$. The first component of $u:=(Q_0L_kQ_0)^{-1}r$ with $r\in Q_0\cH^1$ is given by $u_1=\frac{\ri}{V(\cdot,\omega)}(r_1-\ri k u_3)$, which lies in $\cH^1$ since $r_1, u_3\in \cH^1$ and $1/V_j\in W^{1,\infty}(I_j)$ for all $j$.
	\epf

	\begin{rem}
		As we are dealing with an ODE problem, the unique solvability in Lemma \ref{L:Lk-inv-smooth} can be shown explicitly using the variation of parameters; in Appendix \ref{A:solvable} we provide the corresponding calculation for the example of two homogenous layers. An analogous calculation is, in principle, possible for $N$ layers.
	\end{rem}

	Having satisfied \eqref{phi_GEN}, the rest of \eqref{E:Q-eq} produces the following equation for $(\sigma,\psi) \in \C\times (D(A)\cap \cH^1)$:
	\begin{equation*}
		\begin{split}
			Q_0L_k(\cdot,\omega_0)Q_0\psi=&\eps^{-1} Q_0\big(h(\cdot,\omega,\varphi_0+\eps\phi+\eps^2\psi)-h(\cdot, \omega_0, \varphi_0)\big)+Q_0(\nu\partial_\omega B(\cdot,\omega_0)( \phi+\eps\psi))\\
			&+Q_0\left(\sigma \partial_\omega B(\cdot,\omega_0)+\frac{(\nu+\eps\sigma)^2}{2}\partial_\omega^2B(\cdot,\omega_0) + \eps I(\cdot,\omega)(\nu+\eps\sigma)^3\right)(\varphi_0+\eps \phi+\eps^2\psi)\\
			=:&R(\sigma,\psi),
		\end{split}
	\end{equation*}
	where $\omega=\omega_0+\eps\nu+\eps^2\sigma$. We write this as the fixed point problem
	\beq\label{4_sys}
	\psi = G(\sigma,\psi):= (Q_0L_k(\omega_0)Q_0)^{-1}R(\sigma,\psi).
	\eeq

	The system \eqref{3_sys}, \eqref{4_sys} is solved below for $\sigma$ and $\psi$ via a nested Banach fixed point approach analogously to \cite{dohnal2021eigenvalue}.
	First we solve \eqref{4_sys} for $\psi\in B_{r_2}(0)\subset D(A)$ for each $\sigma\in B_{r_1}(0)\subset\C$ fixed with $r_1>0$ arbitrary and a suitable $r_2=r_2(r_1)$. The radius $r_2$ is chosen so that $G(\sigma,\cdot):B_{r_2}(0)\to B_{r_2}(0)$ is a contraction if $\varepsilon>0$ is small enough.

	Then, having obtained $\psi=\psi(\sigma)\in B_{r_2}(0)$, we shall solve equation \eqref{3_sys} for $\sigma$ in $B_{r_1}(0)\subset \C$  with a suitable $r_1$. Note that the fixed point argument for equation \eqref{3_sys} requires the Lipschitz continuity of $\sigma\mapsto\psi(\sigma)$, which we verify below.

	\begin{rem}
		The alternative approach of solving the system \eqref{3_sys}, \eqref{4_sys} for $(\sigma,\psi)\in B_{r_1}(0)\times B_{r_2}(0)\subset \C\times (D(A)\cap\cH^1)$ via a fixed point argument simultaneously is possible. However, $R(\sigma_1,\psi)-R(\sigma_2,\psi)$ involves the term $Q_0\partial_\omega B(\omega_0)\varphi_0(\sigma_1-\sigma_2)$, which is $O(1)$ and prevents the contraction property. This can be avoided by substituting for $\sigma_1-\sigma_2$ from \eqref{3_sys} but this makes the approach more technical.
	\end{rem}

	To prepare for the fixed point argument, let us start with some estimates on $R$.

	Denote by $c$  a positive constant that	may vary from line to line below but be independent of $\eps$, $r_1$, and $r_2$ for $\eps=\eps(r_1,r_2)>0$ small enough.

	Firstly, for all $|\sigma|\leq r_1$ and  $\|\psi\|_{\cH^1}\leq r_2$, we use Proposition \ref{L:h-lip} to estimate
	\beq\label{E:h-Lip-est}
	\begin{aligned}
		&\|h(\cdot,\omega,\varphi_0+\eps\phi+\eps^2\psi)-h(\cdot, \omega_0, \varphi_0)\|_{\cH^1}\\
		&\qquad \leq\|h(\cdot,\omega,\varphi_0+\eps\phi+\eps^2\psi)-h(\cdot, \omega_, \varphi_0)\|_{\cH^1} + \|h(\cdot,\omega,\varphi_0)-h(\cdot, \omega_0, \varphi_0)\|_{\cH^1}\\
		&\qquad \leq c^{(1)}_a\eps \left(\|\varphi_0+\eps \phi +\eps^2\psi\|_{\cH^1}^2\right)\|\phi+\eps\psi\|_{\cH^1}+c^{(2)}_a\|\varphi_0\|_{\cH^1}^3|\eps\nu+\eps^2\sigma|\\
		&\qquad \leq c(\eps + \eps^2 r_1) + p_3(\eps^2 r_2) \\
		&\qquad \leq c(\eps + \eps^2 (r_1+r_2)),
	\end{aligned}
	\eeq
	where $p_3$ is a cubic polynomial with no zero degree terms. Hence there is a constant $c>0$ such that $p_3(\eps^2 r_2)\leq c \eps^2 r_2$ for all $\eps=\eps(r_2)>0$ small enough, hence the last step is justified.

	 Next, let us define
	 $$M:=\max_{j\in\{1,\dots,m\}} \sup_{\omega\in B_r(\omega_0)}\|V(\cdot,\omega)\|_{W^{1,\infty}(I_j)},$$
	 	where $\|H\|_{W^{1,\infty}(I_j)}$ for a matrix $H$ denotes the maximum of the $W^{1,\infty}(I_j)$-norm of all entries in $H$. We assume w.l.o.g. that $M\geq |\omega_0|+|\delta|$ in order to estimate the entry $\ri \omega$ in $B$ also by $M$ and obtain $\max_{j\in\{1,\dots,m\}} \sup_{\omega\in B_r(\omega_0)}\|B(\cdot,\omega)\|_{W^{1,\infty}(I_j)}\leq M$.

	 Then, for each $j \in \{1,\dots,m\}$
	\begin{align}
		\|(\nu+\varepsilon\sigma)^3I(\cdot,\omega)\|_{W^{1,\infty}(I_j)}&\leq |\nu+\varepsilon\sigma|^3\frac{\max\limits_{z\in \pa B_r(\omega_0)}\|B(\cdot,z)\|_{W^{1,\infty}(I_j)}}{r^{2}(r-|\omega-\omega_0|)} \;\quad \text{for all}\;\; \omega\in B_{r'}(\omega_0) \;\ \text{if} \;\ 0<r'<r<\delta \nonumber\\
		&\leq |\nu+\varepsilon\sigma|^3\frac{2M}{r^3} \quad \text{for all}\quad \omega\in B_{r/2}(\omega_0), \nonumber\\
		&\leq C+q_3(\varepsilon r_1) \leq c(1+\eps r_1),	\label{E:I-est}
	\end{align}
	where $q_3$ is a cubic polynomial with no zero degree terms (such that again $q_3(\eps r_1)\leq c \eps r_1$ for all $\eps=\eps(r_1)$ small enough). In the remaining estimates we directly estimate analogous polynomials of $\eps r_1$ or $\eps^2 r_2$ by the linear terms.

	Due to (A-V) and \eqref{E:I-est}
	$$
	\begin{aligned}
		\|\nu\partial_\omega B(\cdot,\omega_0)( \phi+\eps\psi)+\sigma\partial_\omega B(\cdot,\omega_0)(\varphi_0+\eps\phi+\eps^2\psi)\|_{\cH^1} & \leq c\left(|\sigma|\|\varphi_0\|_{\cH^1}+\left(1+\eps|\sigma|\right)\|\phi\|_{\cH^1}+\eps\left(1+|\sigma|\eps\right)\|\psi\|_{\cH^1}\right)\\
		&\leq c(1+(1+\eps)r_1+(\eps +\eps^2 r_1) r_2)\\
		\|(\nu+\eps\sigma)^2\partial^2_\omega B(\cdot,\omega_0)(\varphi_0+\eps\phi+\eps^2\psi)\|_{\cH^1} & \leq  c(1+\eps r_1+\eps^2r_2)\\
		\|\eps I(\cdot,\omega)(\nu+\eps\sigma)^3(\varphi_0+\eps\phi+\eps^2\psi)\|_{\cH^1} & \leq c\eps|\nu+\eps \sigma|^3\left(\|\varphi_0\|_{\cH^1}+\eps\|\phi\|_{\cH^1}+\eps^2\|\psi\|_{\cH^1}\right)\\
		& \leq c\eps(1+\eps r_1+\eps^2r_2).
	\end{aligned}
	$$
	As a result,
	\beq\label{E:R-est}
	\begin{aligned}
		\|R(\sigma,\psi)\|_{\cH^1}\leq& c_0\left(1+r_1+\eps (r_1+r_2)\right)\leq 2c_0(1+r_1)
	\end{aligned}
	\eeq
	for all $\sigma \in B_{r_1}(0)\subset \C$ and $\psi\in B_{r_2}(0)\subset D(A)\cap \cH^1$. The constant $c_0>0$ is  independent of $\eps, r_1,$ and $r_2$  if $\eps=\eps(r_1,r_2)$ is small enough. The property $G(\sigma,\cdot):B_{r_2}(0)\to B_{r_2}(0)$ is thus satisfied, e.g., with $r_2=r_2(r_1):=2c_0M_L(1+r_1)$ if $\eps>0$ is small enough.

	For the contraction property we first consider  $h$ and use Proposition \ref{L:h-lip} to get
	$$
	\begin{aligned}
	\|h(\cdot,\omega,\varphi_0+\eps \phi + \eps^2 \psi_1)-h(\cdot,\omega,\varphi_0+\eps \phi + \eps^2 \psi_2)\|_{\cH^1} &\leq c^{(1)}_a \eps^2 \sum_{j=1,2}\|\varphi_0+\eps\phi+\eps^2\psi_j\|^2_{\cH^1}\|\psi_1-\psi_2\|_{\cH^1}\\
	& \leq 4c^{(1)}_a\eps^2 \|\varphi_0\|_{\cH^1}^2\|\psi_1-\psi_2\|_{\cH^1}
	\end{aligned}
	$$
	for all $\psi_1,\psi_2\in B_{r_2}(0)$ and $\eps=\eps(r_2)$ small enough. Together with assumption (A-V) and estimate \eqref{E:I-est} this leads to
	$$
	\begin{aligned}
		\|R(\sigma,\psi_1)-R(\sigma,\psi_2)\|_{\cH^1}&\leq c (\eps + (1+r_1)\eps^2) \|\psi_1-\psi_2\|_{\cH^1}\\
		&\leq  c\eps \|\psi_1-\psi_2\|_{\cH^1}
	\end{aligned}
	$$
	for all $\psi_1,\psi_2\in B_{r_2}(0)$, $\sigma \in B_{r_1}(0)$, and for $\eps=\eps(r_1,r_2)$ small enough.

	Given $r_1>0$, we have thus found a unique $\psi=\psi(\sigma)\in B_{r_2}(0)\subset D(A)\cap\cH^1$ for any $\sigma\in B_{r_1}(0),$ where  $r_2=2c_0M_L(1+r_1)$ and $\eps>0$ is small enough.

	We proceed with the second step of the nested argument and solve \eqref{3_sys} for $\sigma$ after substituting $\psi=\psi(\sigma)$. For this we need the Lipschitz continuity of $\sigma\mapsto \psi(\sigma)$ with respect to the $\cH^1$ norm. From $\psi(\sigma)=G(\sigma,\psi(\sigma))$ and Lemma \ref{L:Lk-inv-smooth} we get for $\sigma_1,\sigma_2\in B_{r_1}(0)$
	$$\|\psi(\sigma_1)-\psi(\sigma_2)\|_{\cH^1}\leq M_L \|R(\sigma_1,\psi_1)-R(\sigma_2,\psi_2)\|_{\cH^1},$$
	where $\psi_j:=\psi(\sigma_j), j=1,2$. We denote also $\omega_j:=\omega_0+\eps \nu + \eps^2\sigma_j,$ $u_j:=\eps^{1/2}\varphi_0+\eps^{3/2}\phi+\eps^{5/2}\psi_j$.

	Analogously to \cite{dohnal2021eigenvalue} (see page 15), using \eqref{E:I-est}, one estimates
	\beq\label{E:Idiff}
	\|(\nu+\eps\sigma_1)^3I(\cdot,\omega_1)-(\nu+\eps\sigma_2)^3I(\cdot,\omega_2)\|_{W^{1,\infty}(I_j)} \leq c\eps^4|\sigma_1-\sigma_2| \qquad \forall j \in \{1,\dots,m\},
	\eeq
	where $c$ has a cubic dependence on $r_1$.

	Using again Proposition \ref{L:h-lip} , we get
	\beq
	\begin{split}\label{E:ha-est}
		\|h(\cdot,\omega_1,\eps^{-1/2}u_1)-h(\cdot,\omega_1,\eps^{-1/2}u_2)\|_{\cH^1} &\leq 4c_a^{(1)}\eps^2\|\varphi_0\|_{\cH^1}^2\|\psi_1-\psi_2\|_{\cH^1},\\
		\|h(\cdot,\omega_1,\eps^{-1/2}u_2)-h(\cdot,\omega_2,\eps^{-1/2}u_2)\|_{\cH^1} &\leq 8c_a^{(2)}\eps^2\|\varphi_0\|_{\cH^1}^3|\sigma_1-\sigma_2|
	\end{split}
	\eeq
	for all $\sigma_{1,2}\in B_{r_1}(0)$ and $\eps>0$ small enough. Hence, with the triangle inequalities
	$$
	\begin{aligned}
		&\|h(\cdot,\omega_1,\varphi_0+\eps\phi+\eps^2\psi_1)-h(\cdot,\omega_2,\varphi_0+\eps\phi+\eps^2\psi_2)\|_{\cH^1}\\
		&\leq \|h(\cdot,\omega_1,\eps^{-1/2}u_1)-h(\cdot,\omega_1,\eps^{-1/2}u_2)\|_{\cH^1}+\|h(\cdot,\omega_1,\eps^{-1/2}u_2)-h(\cdot,\omega_2,\eps^{-1/2}u_2)\|_{\cH^1}
	\end{aligned}
	$$
	and
	$$
	\begin{aligned}
		&\|\sigma_1 \pa_\omega B(\cdot,\omega_0)(\varphi_0+\eps\phi+\eps^2\psi_1)- \sigma_2 \pa_\omega B(\cdot,\omega_0)(\varphi_0+\eps\phi+\eps^2\psi_2)\|_{\cH^1} \\
		&\qquad \leq \|\sigma_1 \pa_\omega B(\cdot,\omega_0)(\varphi_0+\eps\phi+\eps^2\psi_1)- \sigma_1 \pa_\omega B(\cdot,\omega_0)(\varphi_0+\eps\phi+\eps^2\psi_2)\|_{\cH^1}\\
		&\qquad +\|(\sigma_1 -\sigma_2)\pa_\omega B(\cdot,\omega_0)(\varphi_0+\eps\phi+\eps^2\psi_2)\|_{\cH^1},
	\end{aligned}
	$$
	and similarly for the other terms in $R$, one obtains
	$$
	\begin{aligned}
		\|R(\sigma_1,\psi_1)-R(\sigma_2,\psi_2)\|_{\cH^1}&\leq c\left(\eps\|\varphi_0\|_{\cH^1}^2\|\psi_1-\psi_2\|_{\cH^1}  +\eps L_a \|\varphi_0\|_{\cH^1}^3|\sigma_1-\sigma_2|\right.\\
		&\quad \left. +\|\varphi_0\|_{\cH^1}|\sigma_1-\sigma_2| + \eps^2\|\psi_1-\psi_2\|_{\cH^1}+ \eps^3|\sigma_1-\sigma_2|\right).
	\end{aligned}
	$$
	This clearly leads to
	\beq\label{E:psi-Lip}
	\|\psi_1-\psi_2\|_{\cH^1}\leq c\|\varphi_0\|_{\cH^1}|\sigma_1-\sigma_2|
	\eeq
	for all $\sigma_{1,2}\in B_{r_1}(0)$ and $\eps>0$ small enough with a modified $c$. Equation \eqref{E:psi-Lip} is the Lipschitz continuity of $\sigma \mapsto \psi(\sigma)$ in $\cH^1$.

	Equation  \eqref{3_sys} is equivalent to
	$$\sigma = S(\sigma),$$
	where
	$$\begin{aligned}
		S(\sigma)&:=c_B\left[\Big\langle\left(\nu\partial_\omega  B(\cdot,\omega_0)\phi+\frac{\nu^2}2\partial_\omega^2B(\cdot,\omega_0)\varphi_0\right),\varphi_0^*\Big\rangle + \varepsilon^{-\frac{5}{2}}\langle v(\cdot,\nu,\sigma,\phi,\psi),\varphi_0^*\rangle\right.\\
		&\quad \left. + \varepsilon^{-\frac{5}{2}}\langle h(\cdot,\omega,u)-\varepsilon^{\frac{3}{2}}h\left(\cdot, \omega_0, \varphi_0\right),\varphi_0^*\rangle \right],\\
		c_B&:=-\langle\partial_\omega B(\cdot,\omega_0)\varphi_0,\varphi_0^*\rangle^{-1}.
	\end{aligned}
	$$
	Next, we find $r_1>0$ such that $S:B_{r_1}(0)\to B_{r_1}(0)$. Using (A-V) and \eqref{E:I-est}, we have for $\eps$ small enough
	$$
	\begin{aligned}
		\|v(\cdot,\nu,\sigma,\phi,\psi)\|_{L^2(\R)}&\leq c_1 \eps^{7/2},\\
		\|h(\cdot,\omega,u)-\varepsilon^{\frac{3}{2}}h\left(\cdot, \omega_0, \varphi_0\right)\|_{L^2(\R)}&\leq c_2\eps^{5/2}\|\varphi_0\|_{\cH^1}^2\|\phi+\eps\psi\|_{\cH^1} + c_3\eps^{5/2}(|\nu|+\eps|\sigma|)\|\varphi_0\|_{\cH^1}^3 + c_4 \eps^{7/2}
	\end{aligned}
	$$
	with $c_1,c_2$ dependent on $r_1$ but independent of $\eps$, and $c_3$ independent of $r_1$ and $\eps$ if $\eps=\eps(r_1)>0$ is small enough. Because $\|\psi\|_{\cH^1} \leq r_2=2c_0M_L(1+r_1)$,  this leads to
	$$|S(\sigma)|\leq 2c_B\left[ c_5+c_2\|\varphi_0\|_{\cH^1}^2\|\phi\|_{\cH^1}+c_3\|\varphi_0\|_{\cH^1}^3(|\nu|+\eps r_1)\right]$$
	for all $\eps$ small enough, where  $c_5:=2\left|\Big\langle\left(\nu\partial_\omega  B(\cdot,\omega_0)\phi+\frac{\nu^2}2\partial_\omega^2B(\cdot,\omega_0)\varphi_0\right),\varphi_0^*\Big\rangle\right|$. Hence, we can choose, e.g., $r_1:= 2c_B(c_5+c_2\|\varphi_0\|_{\cH^1}^2\|\phi\|_{\cH^1}+c_3\|\varphi_0\|_{\cH^1}^3|\nu|+1)$.

	Let us now check the contraction property of $S$ on $B_{r_1}(0)\subset \C$. Using \eqref{E:Idiff} and (A-V), one easily estimates
	$$
	\begin{aligned}
		&|\langle v(\cdot,\nu,\sigma_1,\phi,\psi_1)-v(\cdot,\nu,\sigma_2,\phi,\psi_2),\varphi_0^* \rangle|\\
		&\qquad \leq c \left(\eps^{7/2}|\sigma_1-\sigma_2|+\eps^{9/2}|\sigma_1^2-\sigma_2^2|+\eps^{13/2}|\sigma_1^3-\sigma_2^3|+\eps^{7/2}\|\psi_1-\psi_2\|_{\cH^1}\right).
	\end{aligned}
	$$
	From \eqref{E:ha-est} we obtain
	$$|\langle h(\cdot,\omega_1,u_1)-h(\cdot,\omega_2,u_2),\varphi_0^*\rangle| \leq
	c\eps^{7/2}\left(\|\varphi_0\|_{\cH^1}^2\|\psi_1-\psi_2\|_{\cH^1}+L_a\|\varphi_0\|_{\cH^1}^3|\sigma_1-\sigma_2|\right)
	$$
	with the above definitions of $\omega_1, \omega_2$. Together with \eqref{E:psi-Lip} the last two inequalities lead to
	$$|S(\sigma_1)-S(\sigma_2)|\leq c \eps |\sigma_1-\sigma_2|$$
	for all $\sigma_1,\sigma_2\in B_{r_1}(0)$ and $\eps$ small enough. The contraction property follows for $\eps>0$ small enough and the proof is finished.

	\section{Proof of Theorem \ref{T:PT_sym} (bifurcation under the $\PT$-symmetry)}\label{S:PT-sym}

	We prove next Theorem \ref{T:PT_sym}, i.e. we show that the under the assumption of $\mathcal{P}\mathcal{T}$-symmetry of the coefficients and the realness of $\omega_0$, the above bifurcation argument can be carried out in the $\mathcal{P}\mathcal{T}$-symmetric subspace resulting in real $\omega$ and $\PT$-symmetric $u$.

	As we show now, after multiplication with $-\ri$ the operator $L_k(\cdot,\omega)$ as well as the nonlinearity $h$ commute with the $\PT$ symmetry provided $\hat{\chi}^{(1,3)}$ are $\PT$-symmetric. We have
	\begin{equation}\label{nonlineareigenvalueproblem}
		\tilde{L}_k(\cdot,\omega)u=\tilde{h}(\cdot, \omega, u(\cdot))
	\end{equation}
	for $u=\left(u_1,u_2,u_3\right)^\trans$, where $\tilde{L}_k(x,\omega):=-\ri L_k(x,\omega)$ and $\tilde{h}(x, \omega, u):=-\ri h(x, \omega, u)$. For reference, we recall
	\beq\label{r_operator} \begin{split}\tilde{L}_k(x,\omega)&=\begin{pmatrix}
			0\ & 0 \ &k\\
			0\ & 0 \ &\ri\partial_{x} \\
			-k \ & -\mathrm{i}\partial_x \ & 0
		\end{pmatrix}-\begin{pmatrix}
			V(x, \omega)\ & 0 \ & 0\\
			0\ & V(x, \omega)\ &0 \\
			0 \ & 0 \ & \omega
		\end{pmatrix},\quad V(x,\omega)=-\omega\eps_0\mu_0\left(1+\hat{\chi}^{(1)}(x,\omega)\right),\\
		\tilde{h}(x, \omega, u)&=-\omega\epsilon_0\mu_0^3\hat{\chi}^{(3)}(x,\omega)\left(2|u_E|^2u_E+(u_E\cdot u_E)\overline{u_E}\right)\quad \text{with}\quad u_E=\left(u_1, u_2, 0\right)^\trans.
	\end{split}\eeq
	\blem\label{Lem_PT}
	Let $\omega \in \R$. If the functions $\hat{\chi}^{(1)}(x, \omega)$ and $\hat{\chi}^{(3)}(x, \omega)$ are $\PT$-symmetric, i.e.
	$$\hat{\chi}^{(1)}(x, \omega) = \overline{\hat{\chi}^{(1)}(-x, \omega)}, \ \hat{\chi}^{(3)}(x, \omega) = \overline{\hat{\chi}^{(3)}(-x, \omega)} \ \forall x\in \R,$$
	then the operator $ \tilde{L}_k( \omega)$ and the nonlinearity $\tilde{h}(\cdot, \omega, u)$ commute with the $\PT$-symmetry, i.e.
	\[	\B \tilde{L}_k(x,\omega)=\tilde{L}_k(x,\omega)\B ,\quad \B \tilde{h}(x, \omega, u(x))=\tilde{h}(x, \omega, \B u(x))\quad \text{with}\quad \B u(x):=\overline{u(-x)}.\]
	\elem
	\begin{proof}
		Let $u=(u_1, u_2, u_3)^T\in D(A)$. Due to the symmetry of $V(x, \omega)$ we get
		\begin{equation*}
			\begin{split}
				\B \tilde{L}_k(x,\omega)u&=\begin{pmatrix}
					k\overline{u_3(-x)}-\overline{V(-x, \omega)}\ \overline{u_1(-x)}\\
					-\mathrm{i}(\partial_{x}\overline{u_3})(-x)-\overline{V(-x, \omega)}\ \overline{u_2(-x)}\\
					-k\overline{u_1(-x)}+\ri(\partial_{x}\overline{u_2})(-x)-\omega\overline{u_3(-x)}
				\end{pmatrix}
				=\begin{pmatrix}
					k\overline{u_3(-x)}-V(x, \omega)\ \overline{u_1(-x)}\\
					\mathrm{i}\partial_{x}(\overline{u_3(-x)})-V(x, \omega)\ \overline{u_2(-x)}\\
					-k\overline{u_1(-x)}-\ri\partial_{x}(\overline{u_2(-x)})-\omega\overline{u_3(-x)}
				\end{pmatrix}\\
				&=\tilde{L}_k(x,\omega)\B u.
			\end{split}
		\end{equation*}
		Hence $\B \tilde{L}_k(\omega)=\tilde{L}_k(\omega)\B$.

		Next, we show the $\PT$-symmetry of the nonlinearity $\tilde{h}$. As $\hat{\chi}^{(3)}(x, \omega)$ is $\PT$-symmetric, one, indeed, obtains
		\begin{equation*}
			\begin{split}
				\B (\tilde{h}(\cdot, \omega, u))(x)&=-\omega\eps_0\mu_0^3\overline{\hat{\chi}^{(3)}(-x,\omega)}\ \left(\overline{2|u_E|^2u_E+(u_E\cdot u_E)\overline{u_E}}\right)(-x)\\
				&=-\omega\eps_0\mu_0^3\hat{\chi}^{(3)}(x,\omega)\ \left( 2|u_E(-x)|^2\overline{u_E(-x)}+(\overline{u_E(-x)}\cdot \overline{u_E(-x)})u_E(-x)\right)\\
				&= \tilde{h}(\cdot, \omega, \B u)(x).
			\end{split}
		\end{equation*}
	\end{proof}
	Lemma \ref{Lem_PT} allows us to restrict the bifurcation problem to the $\PT$-symmetric subspace of $D(A)\cap \cH_1$, i.e. to $\{u\in D(A)\cap \cH_1: \B u = u\}$ provided the $\PT$-symmetry of  $\hat{\chi}^{(1,3)}(\cdot, \omega)$ holds for all $\omega$ in a real neighbourhood of $\omega_0$. The bifurcating eigenvalue is then real and the solution $u$ is $\PT$-symmetric, i.e. Theorem  \ref{T:PT_sym} holds. This is proved completely analogously to Proposition 3.1 in \cite{dohnal2021eigenvalue}.

	Note that the proof uses the $\PT$-symmetry of the eigenfunction $\varphi_0$. Under the assumption of algebraic simplicity and realness of $\omega_0$ the eigenfunction can always be selected $\PT$-symmetric. Indeed,
	applying $\B$ to the eigenvalue equation yields $L_k(\omega_0)(\B \varphi_0)=0$ due to the $\PT$-symmetry of $L_k$. As $\omega_0$ is simple, we get $\varphi_0=\B \varphi_0$ (up to a multiplicative constant).


	\section{Numerical Results}\label{S:numerics}

We restrict the numerical computations to two layers ($m=2$) with the interface at $x=0$. We first reduce the nonlinear system \eqref{system} with $\omega \neq 0$ to a system of two equations for $\tilde{u}:=(u_1,u_2)$. For $\omega \neq 0$ the third equation in \eqref{system} namely produces
\beq\label{E:u3}
u_3=-\frac{\ri}{\omega}\left(u_2'-\ri k u_1\right)
\eeq
and the remaining equations become
\begin{align}
	\tilde{u}_2'-\ri \left(k+\frac{\omega V(x)}{k}\right)\tilde{u}_1-\frac{\omega}{k}h_1(x,\omega,\tilde{u})&=0, \quad x\in \R\setminus \{0\}, \label{E:red-syts-a}\\
	\tilde{u}_2''-\ri k\tilde{u}_1 -\omega V(x) \tilde{u}_2+\ri \omega h_2(x,\omega,\tilde{u})&=0, \quad x\in \R\setminus \{0\}.\label{E:red-syts-b}
\end{align}
Due to equation \eqref{E:u3} the interface conditions \eqref{IFCs} become
\beq\label{IFCs-red}
\llbracket \tilde{u}_2 \rrbracket =\llbracket \ri k \tilde{u}_1 - \tilde{u}_2'\rrbracket=0 \ \text{  at  } x=0.
\eeq
As \eqref{E:red-syts-a}-\eqref{IFCs-red} is to be solved for $(\tilde{u}_1,\tilde{u}_2)$ and $\omega$, an additional constraint is needed. We choose the condition
\beq\label{E:normaliz}
\langle u, \varphi_0^*\rangle =\sqrt{\eps}
\eeq
with a given $\eps>0$. By varying $\eps$ in a right neighbourhood of zero we will thus generate a bifurcation curve in accordance with Theorem \ref{mainthm_loc_bifurc}. System \eqref{E:red-syts-a}-\eqref{E:normaliz} is solved via the Newton iteration in a finite difference discretization. Because the function $|\tilde{u}|^2\tilde{u}$ is not complex differentiable, we work in the real variables
$$\tilde{u}_{1,R}, \tilde{u}_{2,R}, \tilde{u}_{1,I}, \tilde{u}_{2,I}, \omega_R, \ \text{and} \ \omega_I,$$
where $\tilde{u}_{j,R}=\Real (\tilde{u}_j), \tilde{u}_{j,I}=\Imag (\tilde{u}_j), j=1,2$, $\omega_R = \Real (\omega)$, and $\omega_I = \Imag (\omega)$. We rewrite \eqref{E:red-syts-a}-\eqref{E:normaliz}  as a system of six real equations with four real interface conditions, namely as the real and imaginary parts of equations \eqref{E:red-syts-a}-\eqref{E:normaliz}.

\subsection{Numerical Finite Difference Method}

The finite difference discretization is implemented on the interval $[-L,L]$ with $L>0$ large enough and with homogenous Dirichlet boundary conditions at $x=\pm L$, which are suitable as we search for localized solutions.

Next, we explain that while equation \eqref{E:red-syts-b} is to be solved at all grid points excluding $x=0$, equation \eqref{E:red-syts-a} has to be solved also in the limit $x\to 0+$ and $x\to 0 -$. The reason is the condition $\nabla \cdot D =0$. The divergence condition
$$
0=\nabla \cdot D =-\frac{1}{\omega}\left[\partial_x\left(V\tilde{u}_1-\ri h_1(x,\omega,\tilde{u})\right) +\ri k \left(V\tilde{u}_1-\ri h_2(x,\omega,\tilde{u})\right)\right]
$$
is satisfied on $\R\setminus \{0\}$ by solutions of \eqref{E:red-syts-a}, \eqref{E:red-syts-b} because these equations imply
$$V\tilde{u}_1-\ri h_1 = \frac{\ri k}{\omega}\left(\ri k \tilde{u}_1-\tilde{u}_2'\right), \quad V\tilde{u}_2-\ri h_2=\frac{1}{\omega}\left(\tilde{u}_2''-\ri k \tilde{u}_1\right).$$
To get $\nabla \cdot D =0$ distributionally on $\R$ (which is the correct formulation of the divergence condition for $u\in D(A)$), we need to satisfy also $\llbracket D_1\rrbracket=0,$ i.e. $\llbracket V\tilde{u}_1-\ri h_1\rrbracket=0$. This follows automatically from the limits $x\to 0\pm$ of \eqref{E:red-syts-a} and the second interface condition in \eqref{IFCs-red}. After discretization it means that we need to solve \eqref{E:red-syts-a} also in the limits $x\to 0-$ and $x\to 0+$.

Choosing $\Delta x :=\frac{2L}{N+1}$ with $N\in 2\N +1$, we have the $N$ grid points $-\frac{N-1}{2}\Delta x,-\frac{N-2}{2}\Delta x,\dots -\Delta x, 0, \Delta x, \dots, \frac{N-1}{2}\Delta x$. We denote $x_j:=-\frac{N-1}{2}\Delta x+(j-1)\Delta x$, $j =1,\dots, N$. Let $j_*:=\frac{N+1}{2}$ be the index of the interface grid point, i.e. $x_{j_*}=0$.

Because $\tilde{u}_1$ is discontinuous at $x=0$, the degrees of freedom of the discretized problem must include approximations of $\tilde{u}_1(0-)$ and $\tilde{u}_1(0+)$. Hence, we have the (real) degrees of freedom
$$U^{R/I}_{1,j}, U^{R/I}_{2,j}, \quad x\in \{1,\dots, j_*-1, j_*+1, \dots, N\}$$
approximating $\tilde{u}_{1,R/I}$ at $x_j, j \neq j_*$ and
$$U^{R/I}_{1,j_*,-}, U^{R/I}_{1,j_*,+}, U^{R/I}_{2,j_*},$$
where $U^{R/I}_{1,j_*,\pm}$ approximate $\tilde{u}_{1,R/I}(0\pm)$ and $ U^{R/I}_{2,j_*}$ approximate $\tilde{u}_{2,R/I}(0)$. Finally, there are the two degrees of freedom $\omega_R$ and $\omega_I$ approximating $\Real (\omega)$ and $\Imag(\omega)$.

These are altogether $4N+4$ real degrees of freedom. Solving \eqref{E:red-syts-a} in the finite difference discretization at $x_j, j\neq j_*$ as well as at $x_1\to 0\pm$ and \eqref{E:red-syts-b} at $x_j, j\neq j_*$, we get $2N+2$ real equations from \eqref{E:red-syts-a} and $2N-2$ equations from \eqref{E:red-syts-b}. Condition \eqref{E:normaliz} produces  two real equations. After using the interface condition $\llbracket \ri k \tilde{u}_1 - \tilde{u}_2'\rrbracket=0$, two degrees of freedom, e.g. $U^R_{1,j_*,+}$ and $U^I_{1,j_*,+}$ are eliminated and we get $4N+2$ real degrees of freedom and $4N+2$ real equations.

At $x_j,j\neq j_*$, we use the centered finite difference stencil of second order for $\partial_x$ as well as for $\pa_x^2$. At $x=0-$ and $x=0+$ we use the one-sided second order stencils. The integral in the normalization condition \eqref{E:normaliz} is approximated using the trapezoidal rule.

\subsection{Bifurcation Examples}

We present here two numerical examples of bifurcation in the case of 2 layers. One of the examples is $\PT$-symmetric and the linear frequency $\omega_0$ (hence also the bifurcating frequency $\omega$) is real. The other example is non-symmetric and the frequency is complex.

The nonlinear equations, discretized using the above finite difference scheme, are solved using the Newton iteration. The initial guess for $\omega = \omega_0+\eps\nu$ with $0<\eps\ll 1$ is provided by the asymptotic approximation $\eps^{1/2}\varphi_0$, see Theorem \ref{mainthm_loc_bifurc}. The bifurcation curve is obtained by a simple parameter continuation in $\eps$.

\subsubsection{$\PT$-symmetric example}
We choose $k=1$ and a case of a $\PT$-symmetric Drude material that is homogenous on each layer $x<0$ and $x>0$. In detail,
\beq\label{E:Vpm-ex}
V_-(\omega)=-\omega \left(1-\frac{2\pi \omega_p^2}{\omega^2 +\ri \gamma \omega}\right), \quad V_+ = \overline{V_-}
\eeq
with $\omega_p = 0.5$ and $\gamma =0.7$. The resulting function $V$ is plotted in Figure \ref{F:V-bif-diag-PT} (a).

We find a real eigenvalue $\omega_0$ by determining real elements of $N^{(k)}$. In the case of \eqref{E:Vpm-ex} condition \eqref{E:ev.cond-1D} can, in fact, be solved explicitly. First, we note that \eqref{E:ev.cond-1D} is equivalent to $k^2(W_+(\omega)+W_-(\omega))=W_+(\omega)W_-(\omega)$. This equation, with $W\pm =-\omega V_\pm(\omega)$ and $V_\pm$ given in \eqref{E:Vpm-ex}, reduces to a quadratic equation in $\mu:=\omega^2$. The two roots are
$$\mu_{1,2}:=k^2+2\pi \omega_p^2 -\frac{\gamma^2}{2}\pm \frac{1}{2}\left(\gamma^2(\gamma^2-4(2\pi \omega_p^2-k^2))+4k^4\right)^{1/2}$$
and we get $\mu_1=\omega_0^2$ with $\omega_0 \approx 1.791$. The corresponding eigenfunction $\varphi_0$ is plotted in Figure \ref{F:u-PT}.

In the nonlinearity $h$ we choose $\epsilon_0\mu_0^3\chi^{(3)}\equiv 1,$ i.e. $\chi^{(3)}$ is real and $x-$ as well as $\omega-$independent. In particular, it is also $\PT$ symmetric.

The discretization parameters are $L=120$ and $N=17999$ (resulting in $\Delta x \approx 0.013$). The numerically approximated value of $\nu$ is $\approx -0.2572$.

In Figure \ref{F:V-bif-diag-PT} (b) we plot the bifurcation diagram showing that the bifurcation parameter (i.e. frequency) $\omega$ indeed stays real. As expected, the bifurcation curve is tangent to the line given by the asymptotic approximation $\omega = \omega_0+\nu \eps$ with $\nu$ from \eqref{E:nuA}. In Figure \ref{F:V-bif-diag-PT} (c) we show that the convergence of the asymptotic approximation error $|\omega-\omega_0 -\eps \nu|$ is indeed approximately quadratic. The solution $u$ (chosen at $\omega\approx 1.7167$) plotted in \ref{F:u-PT} satisfies the $\PT$-symmetry. It is clearly close to the linear solution $\varphi_0$ but not identical.
\begin{figure}[ht!]
	\centering
	\begin{subfigure}[b]{0.27\linewidth}
		\includegraphics[width=\linewidth]{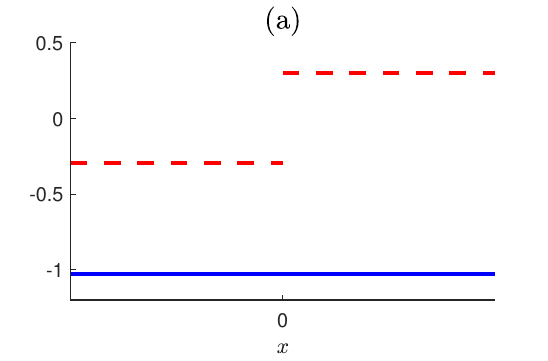}
	\end{subfigure}
	\begin{subfigure}[b]{0.34\linewidth}
	\includegraphics[width=\linewidth]{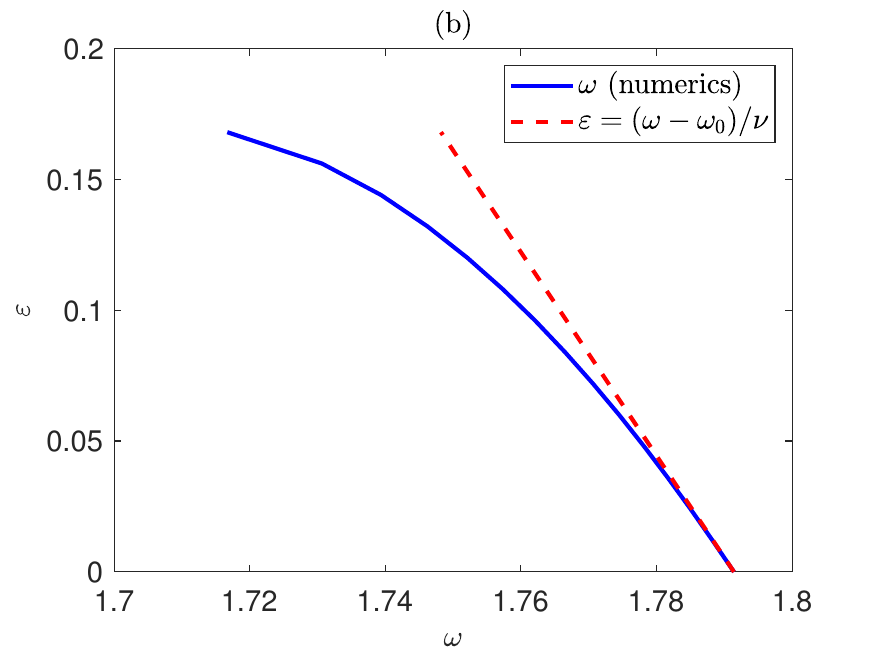}
\end{subfigure}
	\begin{subfigure}[b]{0.34\linewidth}
	\includegraphics[width=\linewidth]{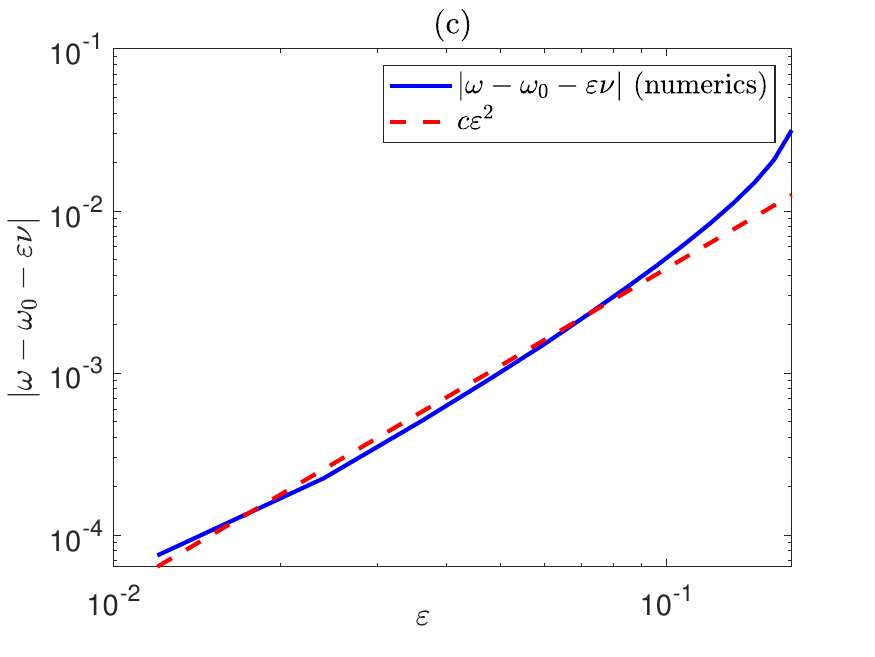}
\end{subfigure}
		\caption{{\footnotesize (a) The graph of the $\PT$-symmetric $V(\cdot,\omega_0)$ given by \eqref{E:Vpm-ex} with $\omega=\omega_0 \approx 1.7914$. (b) The corresponding bifurcation diagram for the bifurcation from the eigenvalue $\omega_0$. (c) Convergence of the approximation error $|\omega-\omega_0-\eps \nu|$. For comparison a curve with a quadratic convergence is plotted.}}
	\label{F:V-bif-diag-PT}
\end{figure}

\begin{figure}[ht!]
	\centering
		\begin{subfigure}[b]{1.05\linewidth}
	\includegraphics[width=\linewidth]{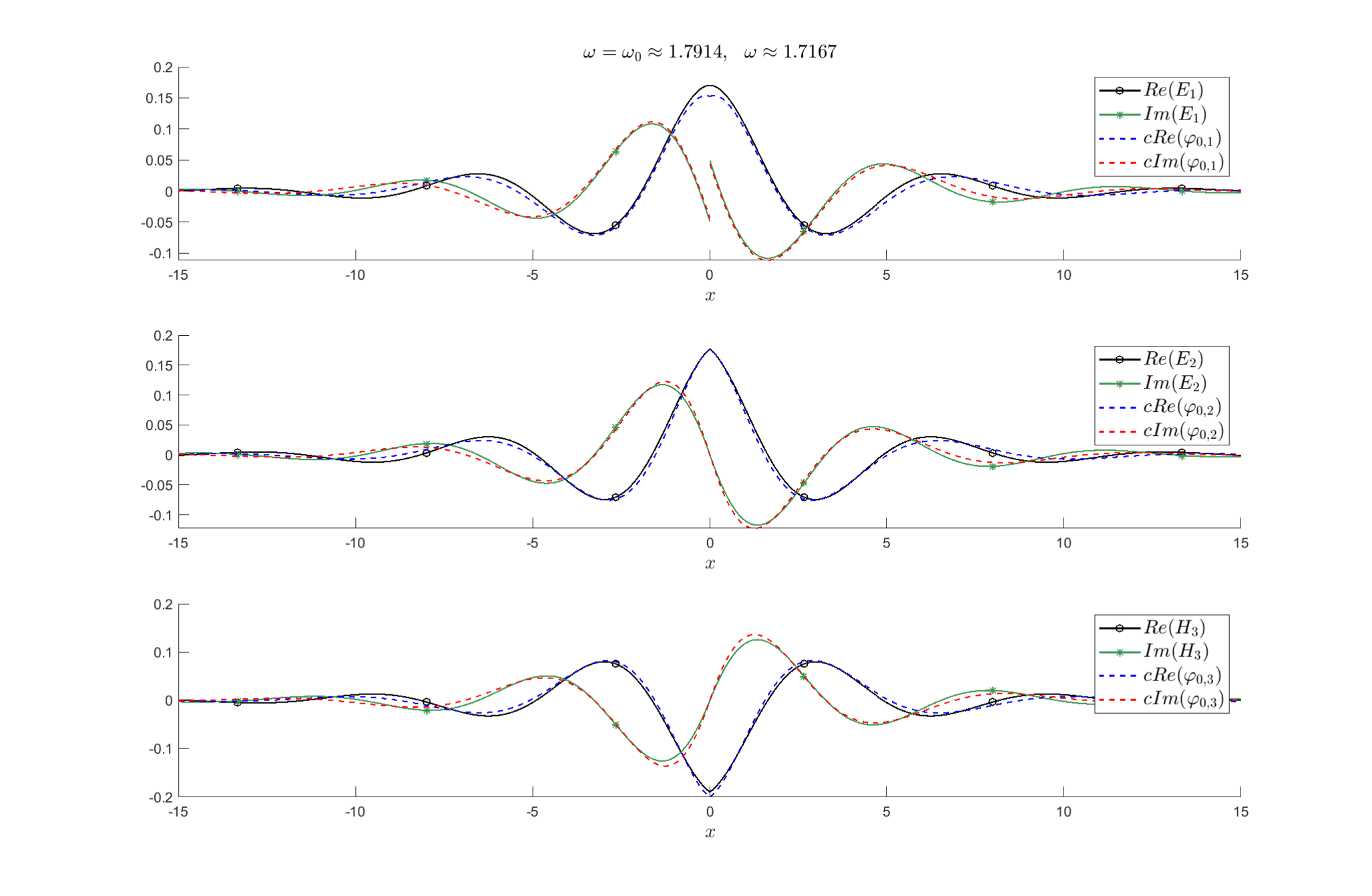}
	\end{subfigure}
	\caption{{\footnotesize The nonlinear solution $u$ at $\omega \approx 1.7167$. Recall that $E_1=u_{1}, E_2=u_{2}, H_3=u_{3}$. The linear eigenfunction $\varphi_0$ (normalized to have a similar amplitude to that of $u$) is plotted for comparison.}}
		\label{F:u-PT}
	\end{figure}

\subsubsection{Non-$\PT$-symmetric example}

We present also one example which is not $\PT$-symmetric, namely
\beq\label{E:Vpm-ex-2}
V_-(\omega)=-\omega \left(1-\frac{2\pi \omega_p^2}{\omega^2 +\ri \gamma \omega}\right), \quad V_+ = - \omega(1+\eta)
\eeq
with $\omega_p=0.8, \gamma=1$, and $\eta=1$. Because of the lack of $\PT$-symmetry, the linear eigenvalues are not real. Equation \eqref{system} makes sense also for non-real $\omega$. However, as explained in Remark \ref{rem:cplx-om}, the corresponding solution $u$ does not generate a solution of Maxwell's equations. Nevertheless, we compute here the bifurcation from $\omega_0\in \C\setminus \R$.

In the nonlinearity $h$ we choose again $\epsilon_0\mu_0^3\chi^{(3)}\equiv 1.$ The discretization parameters are $L=100$ and $N=11999$. The value of $\nu$ is approximated as $\nu \approx  -0.0336 - \ri 0.0054.$

Solving \eqref{E:ev.cond-1D} numerically, we obtain an eigenvalue $\omega_0 \approx 0.4679-\ri~0.061.$ In Fig. \ref{F:V-bif-diag-nonPT} (a) we plot $V$ given by
\eqref{E:Vpm-ex-2} with $\omega=\omega_0$. In (b) we show the bifurcation diagram and the first order asymptotic approximation $\omega = \omega_0+\nu \eps$. The corresponding asymptotic error is plotted in (c) with an observed approximately quadratic convergence, as predicted.
\begin{figure}[ht!]
	\centering
	\begin{subfigure}[b]{0.27\linewidth}
		\includegraphics[width=\linewidth]{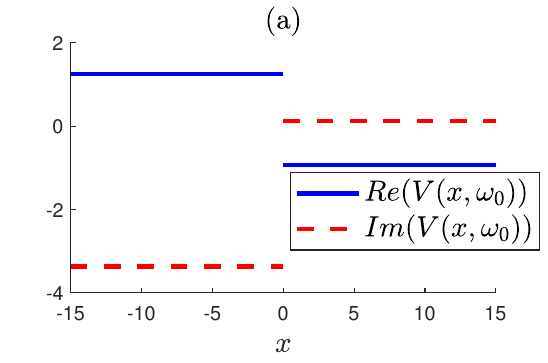}
	\end{subfigure}
	\begin{subfigure}[b]{0.34\linewidth}
		\includegraphics[width=\linewidth]{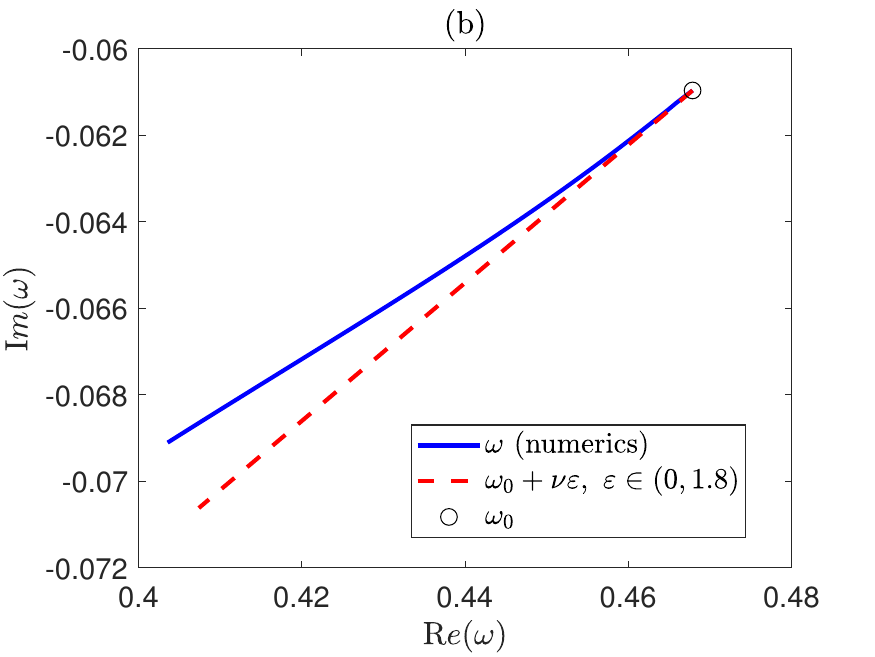}
	\end{subfigure}
	\begin{subfigure}[b]{0.34\linewidth}
		\includegraphics[width=\linewidth]{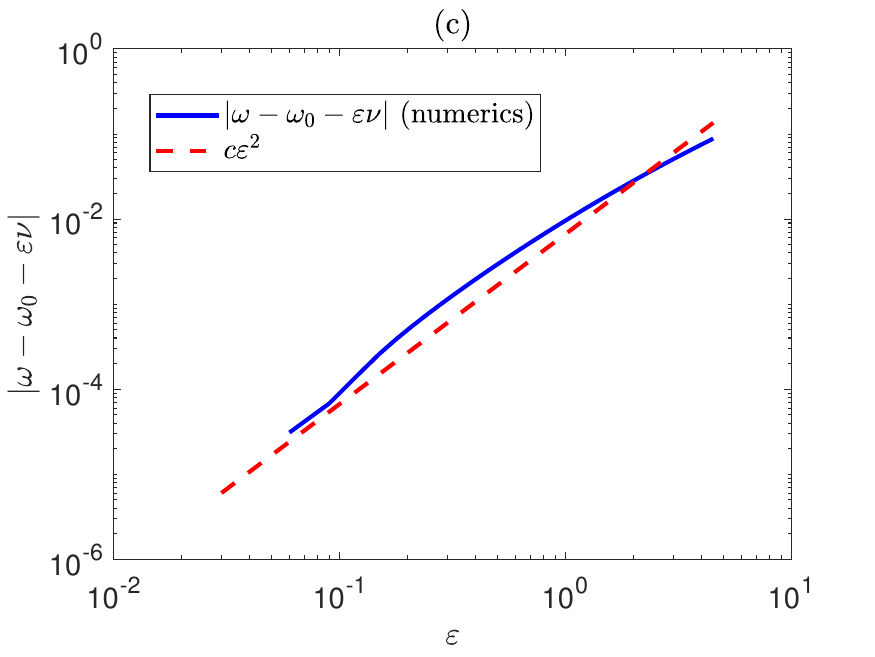}
	\end{subfigure}
	\caption{{\footnotesize (a)  The graph of the $V(\cdot,\omega_0)$ given by \eqref{E:Vpm-ex-2} with $\omega=\omega_0 \approx 0.4679-\ri~0.061$. (b) The corresponding bifurcation diagram for the bifurcation from the eigenvalue $\omega_0$. (c) Convergence of the approximation error $|\omega-\omega_0-\eps \nu|$. For comparison a curve with a quadratic convergence is plotted.}}
	\label{F:V-bif-diag-nonPT}
\end{figure}

The nonlinear solution $u$ at $\omega\approx 0.427-\ri~0.066$ is plotted in (b) together with  the linear eigenfunction $\varphi_0$ normalized to have a similar amplitude to that of $u$.
\begin{figure}[ht!]
	\centering
	\begin{subfigure}[b]{1.05\linewidth}
		\includegraphics[width=\linewidth]{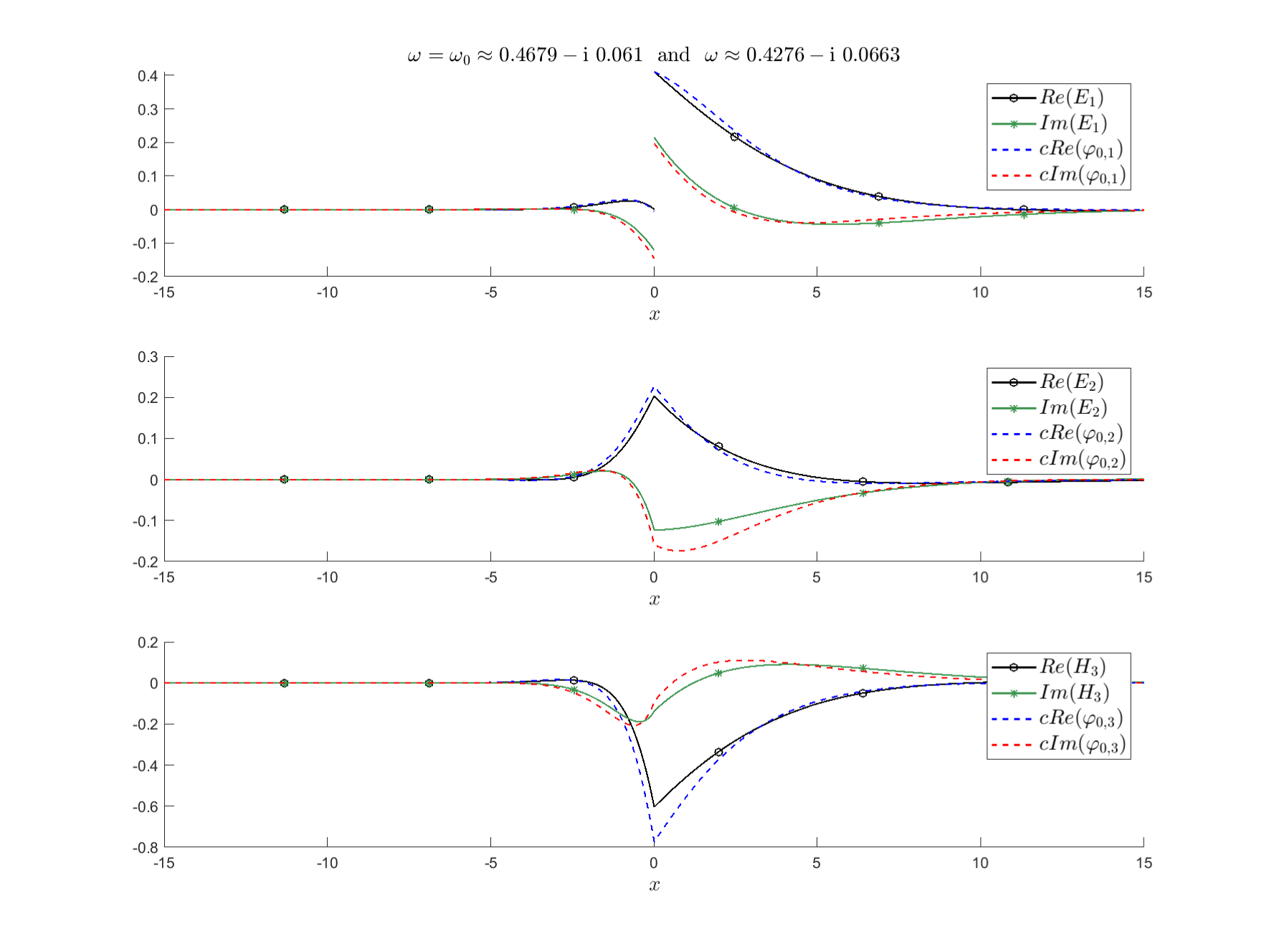}
	\end{subfigure}
	\caption{{\footnotesize The nonlinear solution $u$ at $\omega \approx 0.4276 - \ri 0.066$. Recall that $E_1=u_{1}, E_2=u_{2}, H_3=u_{3}$. The linear eigenfunction $\varphi_0$ (normalized ot have a similar amplitude to that of $u$) is plotted for comparison.}}
	\label{F:u-non-PT}
\end{figure}

Note that in Figure \ref{F:u-non-PT}, both the real and imaginary part of the $E_1$ component are discontinuous, whereas in Figure \ref{F:u-PT} the real part of $E_1$ becomes continuous due to $\PT$-symmetry (real part is even).

	\newpage
	\appendix

	\section{Solvability of $L_k(\omega_0)u=r$ with $\omega_0\in \sigma_p(L_k)$ and $r\in Q_0 L^2(\R)$ for the case of two homogenous layers using the variation of parameters}\label{A:solvable}

	Here we consider the linear case of two homogenous layers of Section \ref{S:2layers}, i.e. $m=2$, $V_1(x,\omega)=V_-(\omega), V_2(x,\omega)=V_+(\omega)$ with the interface at $x=0$. Let $\omega_0 \in \sigma_p(L_k)$ be simple and $\varphi_0\in D(A)$, $\varphi_0^*\in D(A^*)=D(A)$ be corresponding eigenfunctions, i.e. $L_k(\cdot,\omega_0)\varphi_0=0, L_k^*(\cdot,\omega_0)\varphi_0^*=0$. As we know, using the closed range theorem, equation
	$$L_k(x,\omega)u=r \quad \text{ with } \quad \langle r, \varphi_0^*\rangle=0 $$
	has a unique solution in $D(A)\cap \langle \varphi_0^*\rangle^\perp$. Here we want to demonstrate this explicitly using the variation of parameters.

	The solution $u\in D(A)$ is given by \eqref{general_solu}, \eqref{C_plus} and
	\beq\label{E:DR-Cmin}
	\left(\mu_-V_+(\omega_0)+\mu_+V_-(\omega_0)\right)C_-=\frac{1}{2}\left(\mu_+V_-(\omega_0)-\mu_-V_+(\omega_0)\right)\int_{-\infty}^{0}\rho_-^{(2)}(s)e^{\mu_-s}\dd s+\mu_+V_+(\omega_0)\int_{0}^{\infty}\rho_+^{(1)}(s)e^{-\mu_+s}\dd s
	\eeq
	with $\rho_+^{(1)}$ and $\rho_-^{(2)}$ in \eqref{E:rhos}.

	Because $\omega_0\in N^{(k)}$, the left hand side of \eqref{E:DR-Cmin} vanishes, see \eqref{E:ev.cond-1D}. The existence of $u\in D(A)$ follows if the right hand side vanishes too. This is shown below to hold if $r\in Q_0 L^2(\R)$, i.e. if $\langle r,\varphi_0^*\rangle=0.$

	From
	$$L_k^*(\cdot,\omega_0)=A-\overline{B(\cdot,\omega_0)}=\overline{L_{-k}(\cdot,\omega_0)},$$
	we conclude
	$$\varphi_0^*(x;k)=\overline{\varphi_0(x;-k)}.$$
	The eigenfunction $\varphi_0(x;k)$ is given by \eqref{E:psi} with $c_+=-\frac{V_-(\omega_0)}{V_+(\omega_0)}c_-,$ $\mu_\pm=\sqrt{k^2-W_\pm(\omega_0)}$, where $V_+(\omega_0)\mu_-+V_-(\omega_0)\mu_+=0.$ Hence,
	$$
	\varphi_0^*(x;k)=\begin{cases}
		\overline{c_-}e^{\overline{\mu_-}x}\bspm-\ri k\\ \overline{\mu_-} \\ \ri \overline{V_-(\omega)}\espm  &\quad \text{for}\ x<0, \\
		\overline{c_+}e^{-\overline{\mu_+}x}\bspm\ri k\\ \overline{\mu_+}\\ -\ri \overline{V_+(\omega)} \espm  &\quad \text{for}\ x>0.
	\end{cases}
	$$
	The assumption $\langle r, \varphi_0^*\rangle=0 $ becomes
	\beq\label{E:rhs-cond}
	\begin{aligned}
		&\int_{-\infty}^0e^{\mu_-x}\left(\ri k r_1(x)+\mu_-r_2(x)-\ri V_-(\omega_0)r_3(x)\right)\dd x \\
		&-\frac{V_-(\omega_0)}{V_+(\omega_0)}\int_0^\infty e^{-\mu_+ x}\left(-\ri k r_1(x)+\mu_+r_2(x)+\ri V_+(\omega_0)r_3(x)\right)\dd x =0.
	\end{aligned}
	\eeq
	The right hand side of \eqref{E:DR-Cmin} equals
	$$
	\begin{aligned}
		\mu_+V_-(\omega_0)&\int_{-\infty}^0\left(\frac{-kr_1(x)}{V_-(\omega_0)\mu_-}+\ri \frac{r_2(x)}{V_-(\omega_0)}+\frac{r_3(x)}{\mu_-}\right)e^{\mu_- x} \dd x \\
		&+ \mu_+V_+(\omega_0)\int_{0}^\infty\left(\frac{kr_1(x)}{V_+(\omega_0)\mu_+}+\ri \frac{r_2(x)}{V_+(\omega_0)}-\frac{r_3(x)}{\mu_+}\right)e^{-\mu_+ x} \dd x
	\end{aligned}
	$$
	and using \eqref{E:rhs-cond}, it simplifies to
	$$\left(\frac{\mu_+V_-(\omega_0)}{\mu_-V_+(\omega_0)}+1\right)\int_{0}^\infty\left(\frac{kr_1(x)}{V_+(\omega_0)\mu_+}+\ri \frac{r_2(x)}{V_+(\omega_0)}-\frac{r_3(x)}{\mu_+}\right)e^{-\mu_+ x} \dd x,$$
	which is indeed zero since $V_+(\omega_0)\mu_-+V_-(\omega_0)\mu_+=0.$
	\section{Lipschitz continuity of the nonlinearity $h$}\label{loc_nonlinearity}
	In the following we use the notation $\lesssim$ to denote the inequality up to a multiplicative constant, which is independent of the variables and functions appearing on the right hand side of the estimate.

	\begin{prop}\label{L:h-lip}
		Let $\delta >0$ be such that $B_\delta(\omega_0)\subset \Omega$ and such that (A-Na) holds. Then the nonlinearity $h$ given by
		$$
			h(x,\omega, u) =
			-\ri\epsilon_0\mu_0^3\omega\hat{\chi}^{(3)}(x,\omega)\left(2|\tilde{u}|^2(\tilde{u}, 0)^\trans+(\tilde{u}\cdot \tilde{u})\left(\overline{\tilde{u}}, 0\right)^\trans\right), \quad \text{where}\quad u=(\tilde{u}, u_3)^\trans, \ \tilde{u}=(u_1, u_2),
		$$
satisfies $ h(\cdot, \omega,u)\in \cH^1$ for all $u\in D(A)\cap \cH^1$ and $\omega\in B_\delta(\omega_0)$ as well as the Lipschitz properties
			\begin{equation}\label{Lip_est_loc_u}
				\left\|h(\cdot,\omega, \psi)-h(\cdot, \omega, \varphi)\right\|_{\cH^1}\leq c_a^{(1)} \left(\left\|\psi\right\|^2_{\cH^1}+\left\|\varphi\right\|^2_{\cH^1}\right)	\left\|\psi-\varphi\right\|_{\cH^1},
			\end{equation}
			for all $\omega\in B_\delta(\omega_0)$  and $\varphi$, $\psi\in D(A)\cap \cH^1$, where
			$$c_a^{(1)}:=\frac{9}{2}\epsilon_0\mu_0^3(|\omega_0|+\delta)\left(\max_{j\in\{1,\dots,m\}}\|\hat{\chi}^{(3)}(\cdot,\omega_0)\|_{W^{1,\infty}(I_j)}+L_a\delta\right),$$
			and
			\begin{equation}\label{Lip_est_loc_ome}
				\left\|h(\cdot, \omega_1, u)-h(\cdot, \omega_2, u)\right\|_{\cH^1} \leq c_a^{(2)} \left\|u\right\|^3_{\cH^1} \left|\omega_1-\omega_2\right|
			\end{equation}
			for all $\omega_1, \omega_2\in B_{\delta}(\omega_0)$ and $u\in D(A)\cap \cH^1$, where
			$$c_a^{(2)}:=3\epsilon_0\mu_0^3\left(L_a(|\omega_0|+2\delta)+ \max_{j\in\{1,\dots,m\}}\|\hat{\chi}^{(3)}(\cdot,\omega_0)\|_{W^{1,\infty}(I_j)}\right).$$
	\end{prop}
	\begin{proof} Due to the algebra property of $\cH^1$ it follows that
		\begin{equation*}
			\begin{split}
				\left\|h(\cdot,\omega, u)\right\|_{\cH^1}&\lesssim \max_{j\in\{1,\dots,m\}}\left\|\omega\hat{\chi}^{(3)}(\cdot,\omega)\right\|_{W^{1,\infty}(I_j)}\left\|2|\tilde{u}|^2\tilde{u}+(\tilde{u}\cdot \tilde{u})\overline{\tilde{u}}\right\|_{\cH^1}\\
				&\lesssim \|u\|^3_{\cH^1}
			\end{split}
		\end{equation*}
		Hence, $h(\cdot, \omega,u)\in \cH^1$ for all $u\in D(A)\cap\cH^1$ and $\omega\in B_\delta(\omega_0)$. For $\varphi$, $\psi\in D(A)\cap\cH^1$ we write $\varphi=\left(\tilde{\varphi}, \varphi_3\right)^\trans$ with $\tilde{\varphi}=(\varphi_1, \varphi_2)^\trans$ and $\psi=(\tilde{\psi}, \psi_3)^\trans$ with $\tilde{\psi}=\left(\psi_1, \psi_2\right)^\trans$ and obtain \eqref{Lip_est_loc_u} as follows:
		\begin{equation*}
			\begin{split}
				&\left\|h(\cdot,\omega, \psi)-h(\cdot,\omega, \varphi)\right\|_{\cH^1}\\
				\leq &~\epsilon_0\mu_0^3\max_{j\in\{1,\dots,m\}}\left\|\omega\hat{\chi}^{(3)}(\cdot,\omega)\right\|_{W^{1,\infty}(I_j)}\left(2\left\||\tilde{\psi}|^2\tilde{\psi}-|\tilde{\varphi}|^2\tilde{\psi}\right\|_{\cH^1}+\left\|(\tilde{\psi}\cdot\tilde{\psi})\overline{\tilde{\psi}}-(\tilde{\varphi}\cdot\tilde{\varphi})\overline{\tilde{\varphi}}\right\|_{\cH^1}\right)\\
				\leq&~\epsilon_0\mu_0^3(|\omega_0|+\delta)\left(\max_{j\in\{1,\dots,m\}}\|\hat{\chi}^{(3)}(\cdot,\omega_0)\|_{W^{1,\infty}(I_j)}+L_a\delta\right)\left(2\|\tilde{\psi}\|^2_{\cH^1}\left\|\tilde{\psi}-\tilde{\varphi}\right\|_{\cH^1}+2\left\|\tilde{\varphi}\right\|_{\cH^1}\left\||\tilde{\psi}|^2-|\tilde{\varphi}|^2\right\|_{\cH^1}+ \right.\\
				& \left. \qquad  \left\|\tilde{\psi}\cdot \tilde{\psi}-\tilde{\varphi}\cdot \tilde{\varphi}\right\|_{\cH^1}\left\|\tilde{\varphi}\right\|_{\cH^1}+\left\|\tilde{\psi}-\tilde{\varphi}\right\|_{\cH^1}\|\tilde{\psi}\|^2_{\cH^1}\right)\\
				\leq&~c_a^{(1)}\left(\left\|\psi\right\|^2_{\cH^1}+\left\|\varphi\right\|^2_{\cH^1}\right)\left\|\psi-\varphi\right\|_{\cH^1}
			\end{split}
		\end{equation*}
		because $|\omega|\leq |\omega_0|+\delta$ and $\|\hat{\chi}^{(3)}(\cdot,\omega)\|_{W^{1,\infty}(I_j)} \leq \|\hat{\chi}^{(3)}(\cdot,\omega_0)\|_{W^{1,\infty}(I_j)}+L_a\delta$ and using
		$$
		\begin{aligned}
		\left\|\tilde{\psi}\cdot \tilde{\psi}-\tilde{\varphi}\cdot \tilde{\varphi}\right\|_{\cH^1} &\leq \left\|\tilde{\psi}\cdot (\tilde{\psi}-\tilde{\varphi})\right\|_{\cH^1} + \left\|\tilde{\varphi}\cdot (\tilde{\psi}-\tilde{\varphi})\right\|_{\cH^1} \leq (\|\tilde{\psi}\|_{\cH^1}+\|\tilde{\varphi}\|_{\cH^1})\left\|\tilde{\psi}-\tilde{\varphi}\right\|_{\cH^1}\\
		\left\||\tilde{\psi}|^2-|\tilde{\varphi}|^2\right\|_{\cH^1} & \leq \tfrac{1}{2}\|(\tilde{\psi}-\tilde{\varphi})\cdot (\overline{\tilde{\psi}}+\overline{\tilde{\varphi}})\|_{\cH_1} +  \tfrac{1}{2}\|(\tilde{\psi}+\tilde{\varphi})\cdot (\overline{\tilde{\psi}}-\overline{\tilde{\varphi}})\|_{\cH_1} \leq (\|\tilde{\psi}\|_{\cH^1} + \|\tilde{\varphi}\|_{\cH^1})\left\|\tilde{\psi}-\tilde{\varphi}\right\|_{\cH^1}.
		\end{aligned}
		$$

		Similarly, using (A-Na), i.e., the Lipschitz assumption on $\chi^{(3)}$, we get \eqref{Lip_est_loc_ome}. Indeed, for $\omega_1, \omega_2\in B_{\delta}(\omega_0)$ and $u\in D(A)\cap \cH^1$,
		\begin{equation*}
			\begin{split}
				\left\|h(\cdot,\omega_1, u)-h(\cdot,\omega_2, u)\right\|_{\cH^1}&\leq \epsilon_0\mu_0^3 \max_{j\in\{1,\dots,m\}}\left\|\omega_1\hat{\chi}^{(3)}_1(\cdot,\omega_1)-\omega_2\hat{\chi}^{(3)}_2(\cdot,\omega_2)\right\|_{W^{1,\infty}(I_j)}\left\|2|\tilde{u}|^2\tilde{u}+(\tilde{u}\cdot \tilde{u})\overline{\tilde{u}}\right\|_{\cH^1}\\
				&\leq c_a^{(2)}\left|\omega_1-\omega_2\right|\left\|u\right\|^3_{\cH^1},
			\end{split}
		\end{equation*}
		where we have used
		$$
		\begin{aligned}
		\left\|\omega_1\hat{\chi}^{(3)}_1(\cdot,\omega_1)-\omega_2\hat{\chi}^{(3)}_2(\cdot,\omega_2)\right\|_{W^{1,\infty}(I_j)} &\leq \max_{l=1,2}\left\|\hat{\chi}^{(3)}_1(\cdot,\omega_l)\right\|_{W^{1,\infty}(I_j)}|\omega_1-\omega_2|\\
		& \quad +\max_{l=1,2}|\omega_l|\left\|\hat{\chi}^{(3)}_1(\cdot,\omega_1)-\hat{\chi}^{(3)}_2(\cdot,\omega_2)\right\|_{W^{1,\infty}(I_j)}\\
			&\leq (\|\hat{\chi}^{(3)}(\cdot,\omega_0)\|_{W^{1,\infty}(I_j)}+L_a\delta)|\omega_1-\omega_2| + (|\omega_0|+\delta) L_a |\omega_1-\omega_2|.
		\end{aligned}
			$$
	\end{proof}

	\bibliographystyle{abbrv}
	\bibliography{Bibliography1}

\end{document}